\documentclass[11pt,reqno]{amsart}
\usepackage[T1]{fontenc}
\usepackage[utf8]{inputenc} 
\usepackage{lmodern}
\usepackage[english]{babel}
\usepackage{microtype}

\usepackage[%
    a4paper,
	left=2.5cm,       
	right=2.5cm,      
	top=3.5cm,        
	bottom=3.5cm,     
	heightrounded,    
	bindingoffset=0mm 
]{geometry}

\usepackage{graphicx} 
\usepackage{float} 

 \usepackage[parfill]{parskip} 
 
\usepackage{booktabs} 
\usepackage{array} 
\usepackage{paralist} 
\usepackage{verbatim} 
\usepackage{subfig} 
\usepackage{mathrsfs}
\usepackage{amssymb}
\usepackage{xcolor}
\usepackage{amsthm}
\usepackage{amsmath,amsfonts,amssymb,esint,hyperref}
\usepackage[noabbrev, capitalize]{cleveref}

\usepackage{autonum}

\usepackage{graphics,color}
\usepackage{enumerate, enumitem}
\usepackage{mathtools,centernot}
\usepackage{cases}
\usepackage{amsrefs}
\usepackage{bbm}
\usepackage{xfrac}

\usepackage{bookmark}

\newtheorem{theorem}{Theorem}[section]
\newtheorem{lemma}[theorem]{Lemma}

\newtheorem{definition}[theorem]{Definition}
\newtheorem{proposition}[theorem]{Proposition}
\newtheorem{remark}[theorem]{Remark}

\numberwithin{equation}{section}

\newcommand{\T}{\ensuremath{\mathbb{T}}}
\newcommand*{\R}{\ensuremath{\mathbb{R}}}

\newcommand{\eps}{\varepsilon}

\newcommand{\ham}{\dot{H}^{-\frac12}}

\newcommand{\mez}{\frac{1}{2}}

\renewcommand{\MR}[1]{} 

\usepackage{color, graphicx}
\usepackage{mathrsfs, dsfont}

\usepackage[]{hyperref}
\hypersetup{
    colorlinks=true,       
    linkcolor=red,          
    citecolor=blue,        
    filecolor=red,      
    urlcolor=cyan           
}

\def\div{\mathop{\rm div}\nolimits}    
 
\newcommand{\be}{\begin{equation}}
\newcommand{\ee}{\end{equation}}

\title{Hamiltonian compactness and dissipation for the  generalized SQG equation in the inviscid limit}

\author[L. De Rosa]{Luigi De Rosa}
\address[L. De Rosa]{Gran Sasso Science Institute, viale Francesco Crispi, 7, 67100 L’Aquila, Italy}
\email{luigi.derosa@gssi.it}

\author[U. K. Yuzbasioglu]{Utku Kemal Yuzbasioglu}
\address[U. K. Yuzbasioglu]{Gran Sasso Science Institute, viale Francesco Crispi, 7, 67100 L’Aquila, Italy}
\email{utkukemal.yuzbasioglu@gssi.it}

\date{\today}


\begin{document}

\begin{abstract}
We consider the dissipative generalized Surface Quasi-Geostrophic  equation with dissipation given by any fractional power of the Laplacian. In the inviscid limit, it is proved that anomalous dissipation of the Hamiltonian is prevented by the strong compactness of the solutions in the lowest norm that makes the nonlinearity well-defined. In fact, only the dynamics at certain frequencies matters.  The argument is quite robust as it applies regardless of the criticality regime and of the presence of a, possibly noncompact, external forcing. This reveals a more general mechanism behind some recent results obtained for the Navier--Stokes and the critical dissipative Surface Quasi-Geostrophic equations. Because of nonuniqueness issues, in our broader context it is important to work with Leray solutions enjoying suitable higher-order bounds. The existence of such solutions is shown and it might be of independent interest. Finally, we prove that the strong compactness is guaranteed for any initial datum with critical integrability, from which global existence of conservative, although Onsager's supercritical, weak solutions of the inviscid problem is deduced. This offers the largest class of initial data for which global existence is known so far, matching with the one considered by Delort at the endpoint.
\end{abstract}

\maketitle

\section{Introduction}
For any $\alpha\in (0,1]$ and $\gamma\in (0,\infty)$  we consider the dissipative Surface Quasi-Geostrophic equation
\begin{equation}\label{g-SQG diss}\tag{gSQG$_\nu$}
\left\{\begin{array}{l}
\partial_t\theta^\nu + u^\nu\cdot\nabla \theta^\nu + \nu (-\Delta)^\gamma \theta^\nu = f^\nu\\ 
u^\nu = \mathcal R^\perp_\alpha \theta^\nu\\
\theta^\nu(\cdot,0)=\theta^\nu_0
\end{array}\right. 
\end{equation}
posed on $\T^2\times [0,\infty)$. Above, $\mathcal R^\perp_\alpha := \nabla^\perp (-\Delta)^{-\alpha}$ is the orthogonal generalized Riesz transform, $\nu\in (0,1)$ is the viscosity parameter, $\theta^\nu_0$ is the initial datum and $f^\nu$ is the external body force. For convenience in the exposition we restrict ourselves to the case\footnote{Whenever the initial datum is only a periodic distribution, the zero-average condition should be interpreted as $\langle \theta^\nu_0,1\rangle =0$. This will not be specified anymore.}
\begin{equation}\label{initial data zero average}
\int_{\T^2}\theta^\nu_0(x)\,dx=0,
\end{equation}
and note that all the results obtained in the current paper generalize to nonzero-average initial data with minor modifications. Whenever $\int_{\T^2}f^\nu (x,t)\,dx=0$ for all $t>0$,  the condition \eqref{initial data zero average} is propagated in time by \eqref{g-SQG diss} and it allows for Sobolev norms of homogeneous type.

By varying\footnote{The case $\alpha=0$ is trivial since the nonlinearity vanishes.} $\alpha\in (0,1]$ one obtains a continuum of PDEs, with a generally quite different behavior depending on the relation between $\alpha$ and $\gamma$, i.e. depending on whether \eqref{g-SQG diss} gives a subcritical, critical or supercritical PDE. The critical threshold for \eqref{g-SQG diss} is $\alpha=1-\gamma$, giving global well-posedness\footnote{This was first obtained in the influential works \cites{CaffarelliVasseur2010,KNV07} for $\alpha=\gamma=\frac12$. Then, a ``slightly supercritical'' case was settled in \cites{DKV12,DKSV14}.} for sufficiently smooth initial data \cites{MX12,K10}.  Consequently, $\alpha<1-\gamma$ defines the supercritical regime while $\alpha>1-\gamma$ the subcritical one. See also \cites{MX11,KN09} and references therein. When $\alpha=1$, the nonlinearity in \eqref{g-SQG diss} coincides with the one of the incompressible Euler equations in vorticity form, while for $\alpha=\frac12$ it gives the one of the Surface Quasi-Geostrophic equation. 

Formally, by taking the $\dot H^{-\alpha}(\T^2)$ inner product (see \cref{S:notation} for the notation) of  the first equation in \eqref{g-SQG diss} with $\theta^\nu$, one obtains 
\begin{equation}\label{dissip en bal}
\frac12 \frac{d}{dt} \|\theta^\nu(t)\|^2_{\dot H^{-\alpha} } + \nu \|\theta^\nu(t)\|^2_{\dot H^{\gamma-\alpha}} = \langle f^\nu (t),\theta^\nu(t)\rangle_{\dot H^{-\alpha}},
\end{equation}
providing a perfect balance of the $\dot H^{-\alpha}(\T^2)$ norm, which we will refer to as the ``Hamiltonian'', when $\nu=0$. The $\dot H^{-\alpha}(\T^2)$ regularity is also the minimal one that makes the nonlinearity in \eqref{g-SQG diss} well-defined (see \cref{S:weak sol}).
It is important to note that writing the balance \eqref{dissip en bal} might be problematic already at the positive viscosity level since, if $\gamma$ is not large enough depending on $\alpha$, the system \eqref{g-SQG diss} may develop singularities in finite time \cite{K10}. We momentarily ignore this technical issue and address it later on in the paper. Understanding the behavior of \eqref{dissip en bal} as $\nu\rightarrow 0^+$ is key towards establishing a mathematical theory of fluid motion in turbulent regimes. In principle, even under the assumption of strong compactness\footnote{This allows to pass to the limit the first and last term in \eqref{dissip en bal} under the weak compactness of $\{f^\nu\}_\nu$.} in $\dot H^{-\alpha}(\T^2)$, the dissipation $\nu \|\theta^\nu(t)\|^2_{\dot H^{\gamma-\alpha}}$ may not vanish, in turn producing the so-called ``anomalous dissipation'' phenomenon. In three-dimensional turbulence, at least in the the classical setting of the incompressible Navier--Stokes equations, anomalous dissipation experimentally arises together with a uniform-in-viscosity power law decay of the energy spectrum \cite{Frisch95} and it is a cornerstone of the best currently available theories from Kolmogorov \cite{K41} and Onsager \cite{O49}. Therefore, in three space dimensions, strong compactness and anomalous dissipation seem to coexist quite well. 

However, for particular values of $\alpha$ and $\gamma$, recent results have shown that strong compactness in the Hamiltonian norm is incompatible with anomalous dissipation for \eqref{g-SQG diss}. For $\alpha=\gamma=1$, i.e. the Navier--Stokes equations, this was first obtained in \cite{LMP21}, while the case $\alpha=\gamma=\frac12$ was addressed in \cite{DLP25}. The work \cite{CLLS16} was the precursor, while further extensions have been later obtained in \cites{DRP24,DeRosaPark2025,ELL_tocome,JLLL24}. See also \cites{CIN18} for results on bounded domains under stronger assumptions. All these results highlight a striking difference between turbulent flows in two and three space dimensions \cite{DrivasTurbulence26}, as well as between solutions of two-dimensional fluid equations obtained via compactness and the ones constructed by convex integration \cites{DaiGiriRadu2024,GR24,IsettLooi2024,X24}. We will discuss this issue more in details after stating the main theorem. With a broad viewpoint, one may recognize the following similarity in \cites{LMP21,DLP25}. There is a Banach space $X$ such that
\begin{itemize}
    \item $X$ makes the nonlinear term in \eqref{g-SQG diss} well-defined in the sense of distributions;
    \item the norm in $X$ is formally conserved in the absence of viscosity;
    \item the strong compactness in $X$ prevents its own anomalous dissipation as $\nu\rightarrow 0^+$.
\end{itemize}
Our first theorem proves that this is indeed a general pattern, from which the main results of \cites{LMP21,JLLL24,DLP25,CLLS16} follow as a direct corollary\footnote{Theorem \ref{T:main diss} does not recover \cites{DRP24,ELL_tocome} since they apply to a setting in which the strong compactness might fail.}. The space $X$ satisfying such properties for \eqref{g-SQG diss} is $\dot H^{-\alpha}(\T^2)$, which, up to a multiplicative constant, completely characterizes the nonlinearity among all active scalar with a $(1-2\alpha)$-homogeneous odd\footnote{A discussion about ``odd vs even'' multipliers can be found in \cites{IV15,S11}.} Fourier multiplier \cite{IM24}.

\begin{theorem}[Compactness \& dissipation]
    \label{T:main diss}
Let $\alpha\in(0,1]$ and $\gamma \in (0,\infty)$ be arbitrary. Let $\{\theta_0^{\nu}\}_\nu$ be a sequence of zero-average initial data and $\{f^\nu\}_\nu$ be a sequence of external forcing with zero spatial average for almost all times. Assume that 
\begin{itemize}
    \item[$(i)$] $\{\theta_0^{\nu}\}_\nu$ is strongly compact in $\dot H^{-\alpha}(\T^2)$;
    \item[$(ii)$] $\{f^{\nu}\}_\nu$ stays bounded in $L^{\frac{\alpha+\gamma}{\gamma}}([0,T];\dot H^{-\alpha}(\T^2))$ if $\alpha\leq \gamma$, while $\{f^{\nu}\}_\nu$ is bounded in $L^\infty([0,T];\dot H^{-\gamma}(\T^2)) $ if $\alpha>\gamma$.
\end{itemize}
For any sequence $\{\theta^\nu\}_{\nu}$ of the weak solutions to \eqref{g-SQG diss} constructed in Proposition \ref{P:leray exist} there holds 
    \begin{equation}
        \{\theta^\nu\}_\nu\subset L^2([0,T];\dot H^{-\alpha}(\T^2)) \text{ strongly compact} \quad \Longrightarrow \quad \lim_{\nu\rightarrow 0^+}\nu \int_0^T\|\theta^\nu(t)\|^2_{\dot H^{\gamma-\alpha}}\, dt=0.
    \end{equation}
    \end{theorem}
In Proposition \ref{P:leray exist} the solutions are constructed by classical compactness methods. However, with respect to the statements that were available in the literature, the weak solutions we  construct here enjoy a suitable apriori estimate which, as described below, plays a fundamental role in our context. Furthermore, the solutions from Proposition \ref{P:leray exist} also enjoy the energy equality \eqref{dissip en bal}, as opposite to the more standard inequality, and therefore are strongly continuous in time. To the best of our knowledge,  this is new at this level or regularity of the data. We emphasize that such solutions are proved to exist for any $\theta^\nu_0\in \dot H^{-\alpha}(\T^2)$. Therefore, the range of applicability of Theorem \ref{T:main diss} is quite wide.

Without the strong compactness of the initial data it is easy to produce anomalous dissipation, even by keeping the space-time compactness of $\{\theta^\nu\}_\nu$ in $L^2([0,T];\dot H^{-\alpha}(\T^2))$. The point is that the dissipation might accumulate all the mass as $t=0$ if $(i)$ does not hold (see Remark \ref{R:initial strong compactness}). We emphasize that the space-time strong compactness of the sequence of solutions might, in principle, be ruined by the wild nonlinear dynamics in time, even under the assumption $(i)$. For $\alpha=1$ the strong compactness has been very recently proved to hold for Baire generic $\dot H^{-1}(\T^2)$ initial data \cite{G26}. It is unfortunate that, although quite foundational, this issue remains to date poorly understood in its full generality.

The sequence $\{f^\nu\}_\nu$ is only assumed to be bounded, therefore not necessarily compact in the corresponding norm, highlighting the robustness of the approach. This also shows that the recent constructions \cites{SJ24,brue2022onsager,BD23,CL21,CP25} of anomalous dissipation for the  forced three-dimensional Navier--Stokes equations cannot extend to the two-dimensional setting, for the whole family \eqref{g-SQG diss}.  The threshold $\alpha=\gamma$ naturally arises when trying to prove the higher-order bound \eqref{hob} for the solution, which is an essential step towards establishing Theorem \ref{T:main diss} (see the heuristic given in \cref{S:heuristic}, with emphasis on Remark \ref{R:main issues}). Note that such threshold has nothing to do with the one determining the criticality of \eqref{g-SQG diss}, which is instead given by $\alpha=1-\gamma$ (see \cite{LazarXue2019}). Curiously, the proof of the higher-order bound in the two regimes $\alpha\leq \gamma$ and $\alpha>\gamma$ is quite different, and none of the two applies to the other case. The case $\alpha>\gamma$ is the most technical one, since there is not enough dissipation to run the proof in the cleanest way and without going through a Gr\"onwall argument. It is because of this technical difficulty that we have required the stronger assumption on the force in the time variable when $\alpha>\gamma$. Note that $\frac{\alpha+\gamma}{\gamma}\leq 2$ if $\alpha\leq \gamma$.

 While the previous works \cites{LMP21,DLP25,JLLL24} were concerned with $\alpha=\gamma$, in our general setting there is some additional difficulty in obtaining the higher-order bounds.  In the case $\alpha=\gamma$ they somehow come more naturally from the PDE, while in the general case considered in Theorem \ref{T:main diss} they require some extra manipulations. Another important difference with respect to the previous literature \cites{LMP21,DeRosaPark2025,JLLL24} on the Navier--Stokes equations is that Leray solutions (in the sense of Definition \ref{def.viscous-sqg-sol})  to \eqref{g-SQG diss} might not be unique \cites{AC23,MS06} for general values of $\alpha$ and $\gamma$. In particular, there might be Leray solutions that do not satisfy the aforementioned higher-order bounds. It is therefore essential to work only with the ones enjoying further suitable estimates (see Proposition \ref{P:leray exist}) since otherwise the conclusion of Theorem \ref{T:main diss} might fail. See Remark \ref{R:main issues} for further discussions.

Theorem \ref{T:main diss} establishes no anomalous dissipation of the Hamiltonian in the inviscid limit conditionally to a uniform-in-viscosity ``Onsager supercritical'' assumption, i.e. below the regularity threshold $L^3([0,T];B^{\frac{2-4\alpha}{3}}_{3,\infty}(\T^2))$, at least for $\alpha\in [\frac12,1)$, above which it is possible to prove the Hamiltonian balance for $\nu=0$ (see \cite{IM24}). In fact, in Theorem \ref{T:compact and diss at frequencies} we will see that ``compactness at certain frequencies'' is already enough to rule out the dissipative anomaly, offering an even more general statement. Nonetheless, the Onsager's conjecture for inviscid active scalar is true \cites{DaiGiriRadu2024,IsettLooi2024,GR24,X24}, establishing the optimality of the Onsager's threshold. These wild solutions are obtained by convex integration and their regularity compactly embeds in $L^2([0,T];\dot H^{-\alpha}(\T^2))$. It follows that none of such solutions can be achieved by physically reasonable\footnote{For instance, the ones constructed by classical compactness methods or simply the smooth ones, whenever they exist.} weak solutions to \eqref{g-SQG diss} in the limit $\nu\rightarrow 0^+$, keeping their regularity uniformly in the viscosity. This is in sharp contrast to the  largely believed three-dimensional picture (see for instance \cite{Is22}*{Conjecture 2}).

\begin{remark}
The reason why strong compactness and anomalous dissipation for \eqref{g-SQG diss} cannot coexists is only due to the aforementioned higher-order bounds. As it is apparent from the sketch given in \cref{S:heuristic}, the structure of the nonlinearity that allows for such bounds is precisely the one that enforces infinitely many conserved quantities in the unforced ideal setting $\nu=0$. It is the very same structure that induces a ``double cascade'', the main distinguishing feature of two-dimensional turbulence \cites{Kraich67,Batch69,BCPW20}. See also \cite{CCFS08} for further differences (mainly about the ``locality'' of the energy flux) between two and three dimensions. Therefore, although it might be a bit too ambitious to state a precise question, it is natural to wonder to what extent the presence of an inverse cascade is the ultimate reason behind Theorem \ref{T:main diss}. For instance, the so-called ``Taylor's conjecture'' is true \cites{FL20,FLMV22}, providing unconditional magnetic helicity conservation for all solutions of the ideal MHD equations obtained in the zero viscosity/resistivity limit, going beyond the corresponding Onsager-type threshold \cites{KL07,FLS24}. Coherently with the above reasoning, the MHD system, although three-dimensional, experiences an inverse cascade\footnote{The first named author got informed about that in a conversation with Gregory Eyink and Theodore D. Drivas during an ``Eyink's Online Turbulence Seminar''.} in the magnetic helicity \cite{AMP06}. 
\end{remark}

In view of Theorem \ref{T:main diss}, it is natural to investigate the largest class of initial data which guarantees the strong compactness of $\{\theta^\nu\}_\nu$ in $L^2_{\rm loc}([0,\infty);\dot H^{-\alpha}(\T^2))$. In addition to provide an effective setting in which no anomalous dissipation can occur, it would also provide a global existence result for weak solutions of the inviscid system. Indeed, the nonlinearity is a continuous bilinear operator on $L^2_{\rm loc}([0,\infty);\dot H^{-\alpha}(\T^2))$ (see Remark \ref{R:nonlinearity continuous}). Our second main result addresses precisely this issue. We denote by $C^0_{\rm w}([0,\infty);L^p(\T^2))$ the space of continuous-in-time functions with values in $L^p(\T^2)$ endowed with the weak topology.

\begin{theorem}[Global existence]\label{T:global-existence}
Let $\alpha\in (0,1)$ be arbitrary and set $p_\alpha:=\frac{2}{1+\alpha}$. Let $\theta_0\in L^{p_\alpha} (\T^2)$ be a zero-average initial datum and $f\in L^1_{\rm loc}([0,\infty);L^{p_\alpha} (\T^2))\cap L^{1+\alpha}_{\rm loc}([0,\infty);\dot{H}^{-\alpha}(\T^2))$ be with zero spatial average for almost all times. There exists a global-in-time weak solution 
$$
\theta\in C^0([0,\infty);\dot{H}^{-\alpha}(\T^2))\cap C^0_{\rm w}([0,\infty);L^{p_\alpha}(\T^2)) 
$$
to 
\begin{equation}\label{g-SQG}\tag{gSQG}
  \left\{\begin{array}{l}
\partial_t\theta + u\cdot\nabla \theta  = f\\ 
u = \mathcal R^\perp_\alpha \theta\\
\theta(\cdot,0)=\theta_0
\end{array}\right. 
\end{equation} 
in the sense of Definition \ref{def.sqg-solutions} such that $\int_{\T^2}\theta(x,t)\,dx=0$, 
\begin{align}\label{hamiltonian balance}
\|\theta (t)\|_{\dot{H}^{-\alpha}}^2= \|\theta_0\|^2_{\dot{H}^{-\alpha}} + 2\int_0^t  \langle f(\tau),\theta(\tau)\rangle_{\dot H^{-\alpha}}\,d\tau 
\end{align}
and 
\begin{equation}\label{Lp inequality}
\|\theta (t)\|_{L^{p_\alpha}}\leq \|\theta_0\|_{L^{p_\alpha}} +\int_0^t \|f(\tau)\|_{L^{p_\alpha}}\,d\tau,
\end{equation}
for all times $t\geq 0$.
\end{theorem}
The solutions are obtained by vanishing viscosity. The main point is to establish the strong compactness of $\{\theta^\nu\}_\nu$ in $L^2_{\rm loc}([0,\infty);\dot{H}^{-\alpha}(\T^2))$, which is needed to pass to the limit in the nonlinear term. Note that the integrability exponent $p_\alpha$ is critical in the sense that the embedding $L^{p_\alpha}(\T^2)\subset \dot{H}^{-\alpha}(\T^2)$ is only continuous. Therefore, the spatial strong compactness in $\dot{H}^{-\alpha}(\T^2)$ cannot be deduced by the compactness of the Sobolev embedding. We overcome this issue by providing a splitting of the solution into an ``arbitrarily small'' and a ``regular'' part. The regular part will be used to kill concentrations into high-frequencies, from which the desired strong compactness is deduced thanks to the arbitrariness in the size of the small part (see Lemma \ref{InterpolatoryLem}). The conservation of the Hamiltonian is then a direct consequence of Theorem \ref{T:main diss} since no dissipation can occur whenever the sequence $\{\theta^\nu\}_\nu$ is strongly compact in $L^2_{\rm loc}([0,\infty);\dot{H}^{-\alpha}(\T^2))$. The inequality \eqref{Lp inequality} comes by the lower semicontinuity of the norm under weak convergence. The assumption on the force $f\in L^{1+\alpha}_{\rm loc}([0,\infty);\dot{H}^{-\alpha}(\T^2))$ can be most likely removed (see Remark \ref{R:no L2 force}). As well as for \eqref{g-SQG diss}, weak solutions to \eqref{g-SQG}  are known to be nonunique \cites{DaiGiriRadu2024,CastroFaracoMengualSolera2025,V1,V2,ABCDGJH,IsettLooi2024,BSV19}.

Theorem \ref{T:global-existence} recovers \cite{DLP25}*{Theorem 1.1} obtained for $\alpha=\frac12$. In addition, besides the $\dot H^{-1}(\T^2)$ Baire generic global existence recently proved in \cite{G26} for $\alpha=1$, it offers the largest class known so far of initial data for which global existence for \eqref{g-SQG} holds. When $\alpha\rightarrow 1^-$ we have $p_\alpha\rightarrow 1^+$, sharply matching the well-known one for the incompressible Euler equations. For $\alpha=1$, global existence is known to hold for measure initial data with singular part of distinguished sign since the work of Delort \cite{delort1991existence} (see also \cites{diperna1987concentrations,DM87,DM88} for earlier results and \cites{evans1994hardy,scho95,vecchi19931,majda1993remarks} for further extensions). In this case, the lack of the continuous embedding $L^{1}(\T^2)\subset \dot{H}^{-1}(\T^2)$ makes the proof substantially more delicate: only some specific quantities pass to limit, and they do it so only in a distributional sense since the strong compactness might, in principle, fail. Therefore, the endpoint $\alpha=1$ requires a quite different argument and, since it is anyway already known, it has been excluded from the statement of Theorem \ref{T:global-existence}.

\subsection*{Organization of the paper} In \cref{S:comp and diss} we will prove Theorem \ref{T:main diss} together its frequency-dependent generalization from Theorem \ref{T:compact and diss at frequencies}. \cref{S:global existence} will be instead dedicated to the proof of Theorem \ref{T:global-existence}.

\section{Compactness vs dissipation}\label{S:comp and diss}
\subsection{Notation}\label{S:notation} Let $g:\T^2\rightarrow \R$. Denoting by $\hat{g}(n)$ its Fourier coefficients we have
$$
g(x)=\sum_{n \in \mathbb{Z}^2} \hat{g}(n) e^{i n\cdot x}.
$$
For simplicity we will only consider functions with zero average. Therefore $\hat g(0)=0$ and we can sum over $n\in \mathbb{Z}^2\setminus\{0\}$ only. 

For any $s\in \R$ we define the homogeneous Sobolev norm by
$$
\|g\|_{\dot{H}^s}^2 := \sum_{n \in \mathbb{Z}^2\setminus\{0\}} |n|^{2s}|\hat{g}(n)|^2,
$$
which is induced by the inner product 
\begin{equation}
    \label{inner prod H - 12}
    \langle g , h \rangle_{\dot H^{s}}:=\sum_{n \in \mathbb{Z}^2\setminus\{0\}} |n|^{2s} \hat{g}(n) \hat{h}(n).
\end{equation}
For any $\alpha \in \R$ and any zero average function $g:\T^2\rightarrow \R$, the fractional Laplacian $(-\Delta)^\alpha g$ is defined as 
$$
(-\Delta)^\alpha g (x)=|\nabla |^{2\alpha }g(x):= \sum_{n \in \mathbb{Z}^2\setminus\{0\}}|n|^{2\alpha} \hat{g}(n) e^{i n\cdot x}.
$$
It follows
$$
\|g\|_{\dot{H}^s}=\|(-\Delta)^{\frac{s}{2}}g\|_{L^2}= \| |\nabla |^{s } g\|_{L^2} \qquad \forall s\in \R.
$$
In particular, if $\langle \cdot,\cdot\rangle$ denotes the standard duality paring, we have 
\begin{equation}
    \label{duality and inner product}
    \langle g,h\rangle_{\dot H^s} =  \langle |\nabla|^s g, |\nabla|^s h\rangle = \langle |\nabla|^{2s} g,  h\rangle = \langle  g,  |\nabla|^{2s} h\rangle\qquad \forall s\in \R.
\end{equation}

\subsection{Main idea behind Theorem \ref{T:main diss}}\label{S:heuristic} The proof of the theorem depends on three main, rather independent, ingredients
\begin{itemize}
    \item[(1)] higher-order bound;
    \item[(2)] dissipation splitting for positive times;
    \item[(3)] no instantaneous dissipation.
\end{itemize}
We now give a sketch of each step. To make it as light as possible,  we will assume that everything is smooth enough to run the computations and that there is no external force, i.e. we assume  $f^\nu = 0$. We postpone to Remark \ref{R:main issues} a more careful discussion about the technicalities which, in view of possible finite-time singularities \cites{K10} and nonuniqueness \cites{AC23,MS06} for \eqref{g-SQG diss}, play in fact an important role in applying the argument to suitably defined sequences of Leray solutions. 

\subsubsection*{$(1)$ Higher-order bound} Under the assumption that $\{\theta^\nu_0\}_\nu\subset \dot H^{-\alpha}(\T^2)$ is bounded,
here the goal is to prove that, for any $\delta>0$, there exists an implicit constant such that 
\begin{equation}
    \label{higher order bound idea}
     \nu^{\frac{\alpha+\gamma}{\gamma}} \int^T_\delta \|\theta^\nu(t)\|_{\dot H^{\gamma}}^2 \, dt  \lesssim 1\qquad \forall \nu>0.
\end{equation}

To prove \eqref{higher order bound idea}, we start with the balance
 \begin{equation}\label{entrophy balance idea}
     \frac{d}{dt} \| \theta^\nu(t) \|_{L^2}^2 = -2\nu \| \theta^\nu(t) \|^2_{\dot H^{\gamma}}.
 \end{equation}
 By the interpolation inequality\footnote{It is here that it is convenient to work with zero-average solutions, since otherwise it is necessary to add a lower-order term to make the interpolation inequality true.} $\| \theta^\nu(t) \|_{L^2}\lesssim \| \theta^\nu(t) \|_{\dot H^{-\alpha}}^{\frac{\gamma}{\alpha+\gamma}}\| \theta^\nu(t) \|^{\frac{\alpha}{\alpha+\gamma}}_{\dot H^{\gamma}}$ , and since  $\| \theta^\nu (t) \| _{\dot H^{-\alpha}}\leq \| \theta^\nu_0 \| _{\dot H^{-\alpha}}$ by \eqref{dissip en bal} (recall we are considering $f^\nu=0$),  we obtain
 \begin{equation}
     \frac{d}{dt} \| \theta^\nu(t) \|_{L^2}^2 \lesssim -\nu \| \theta^\nu(t) \|^{2\frac{\alpha + \gamma}{\alpha}}_{L^2}.
 \end{equation}
The integration of the differential inequality yields to $ \| \theta^\nu(t) \|_{L^2}^2 \lesssim (\nu t)^{-\frac{\alpha}{\gamma}}$. Therefore, by integrating \eqref{entrophy balance idea} in time, we conclude
\begin{equation}
     \nu\int^T_\delta \| \theta^\nu(t)\|_{\dot H^{\gamma}}^2 \, dt \leq \| \theta^\nu(\delta) \|^2_{L^2}\lesssim (\nu \delta)^{-\frac{\alpha}{\gamma}},
\end{equation}
and the higher-order bound \eqref{higher order bound idea} is proved. 

\subsubsection*{$(2)$ Dissipation splitting for positive times} We will split the dissipation into high and low frequencies. To do that,  define the Fourier projections of a periodic function $g$ as
\begin{equation}\label{freq cuts}
    g_{ \leq N}(x) \vcentcolon=\sum_{\substack{n \in \mathbb{Z}^2 \\ |n| \leq  N}} \hat g(n) e^{in \cdot x}\qquad \text{and} \qquad g_{>N}:=g-g_{\leq N}.
\end{equation}
The goal is to prove, for any $\delta>0$, that 
\begin{equation}
    \label{diss splitting bound idea}
     \nu\int^T_\delta \| \theta^\nu(t)\|_{\dot H^{\gamma-\alpha}}^2 \, dt\lesssim \nu N^{2\gamma} + \Phi^{\frac{\alpha}{\alpha+\gamma}} (N)\qquad \forall \nu,N>0,
\end{equation}
with
$$
\Phi (N):=  \sup_{\nu>0}\int^T_\delta \| \theta^\nu_{>N}(t)\|_{\dot H^{-\alpha}}^2 \, dt.
$$
Note that $\Phi(N)\rightarrow 0$ as $N\rightarrow \infty$ if $\{\theta^\nu\}_\nu\subset L^2([0,T];\dot H^{-\alpha}(\T^2))$ is strongly compact. Therefore, under the latter compactness assumption, one can first let $\nu\rightarrow 0^+$ in \eqref{diss splitting bound idea}, and then\footnote{In fact, from \eqref{diss splitting bound idea} it is clear that only frequencies $N\sim \nu^{-\frac{1}{2\gamma}}$ matter. This leads to the frequency-dependent compactness statement of Theorem \ref{T:compact and diss at frequencies}.} $N\rightarrow \infty$, to deduce that there is no anomalous dissipation for all times away from $t=0$, i.e.
\begin{equation}
    \label{no diss for positive times idea}
      \lim_{\nu\rightarrow 0^+}\nu\int^T_\delta \| \theta^\nu(t)\|_{\dot H^{\gamma-\alpha}}^2 \, dt=0 \qquad \forall \delta>0.
\end{equation}
To prove \eqref{diss splitting bound idea} we split
\begin{equation}
    \| \theta^\nu(t) \|^2_{\dot H^{\gamma - \alpha}} =  \| \theta^\nu_{\leq N } (t)\|^2_{\dot H^{\gamma - \alpha}} + \| \theta^\nu_{>N} (t)\|^2_{\dot H^{\gamma - \alpha}}.
\end{equation} 
We bound the low-frequency part with the  Hamiltonian at the initial time as 
$$
\| \theta^\nu_{\leq N } (t)\|^2_{\dot H^{\gamma - \alpha}}\leq N^{2\gamma}\| \theta^\nu (t)\|^2_{\dot H^{- \alpha}}\leq N^{2\gamma} \| \theta^\nu_{0}\|^2_{\dot H^{- \alpha}}\lesssim N^{2\gamma},
$$
while for the high-frequency part we use  interpolation 
$$
\| \theta^\nu_{>N} (t)\|_{\dot H^{\gamma - \alpha}}\lesssim \| \theta^\nu_{>N} (t)\|^{\frac{\alpha}{\alpha+\gamma}}_{\dot H^{- \alpha}}  \| \theta^\nu (t)\|^{\frac{\gamma}{\alpha+\gamma}}_{\dot H^{ \gamma}}.
$$
By plugging these two last estimates into \eqref{diss splitting bound idea}, integrating in time, and using the H\"older inequality on the time integral, we obtain
\begin{align}
     \nu \int^T_\delta \| \theta^\nu (t) \|_{\dot H^{\gamma - \alpha}}^{2} \, dt &\lesssim \nu N^{2\gamma} + \nu\int_\delta^T\| \theta^\nu_{>N} (t)\|^{2\frac{\alpha}{\alpha+\gamma}}_{\dot H^{- \alpha}}  \| \theta^\nu (t)\|^{2\frac{\gamma}{\alpha+\gamma}}_{\dot H^{ \gamma}}\,dt\\
     &\lesssim \nu N^{2\gamma} + \left(\int_\delta^T\| \theta^\nu_{>N} (t)\|^{2}_{\dot H^{- \alpha}} \,dt\right)^\frac{\alpha}{\alpha+\gamma} \left(\nu^{\frac{\alpha+\gamma}{\gamma}}\int_\delta^T\| \theta^\nu (t)\|^{2}_{\dot H^{ \gamma}} \,dt\right)^\frac{\gamma}{\alpha+\gamma}\\
     &\lesssim \nu N^{2\gamma} + \Phi^{\frac{\alpha}{\alpha+\gamma}}(N),
\end{align}
where the last inequality follows by the higher-order bound \eqref{higher order bound idea}. We have proved \eqref{diss splitting bound idea}.

\subsubsection*{$(3)$ No instantaneous dissipation} The goal of this last step is to prove that 
\begin{equation}
    \label{no instant dissipation idea}
    \{\theta_0^\nu\}_\nu\subset \dot H^{-\alpha}(\T^2) \text{ strongly compact} \quad \Longrightarrow \quad \lim_{\delta\rightarrow 0^+} \sup_{\nu>0} \nu\int_0^\delta \| \theta^\nu (t) \|_{\dot H^{\gamma - \alpha}}^{2} \, dt =0,
\end{equation}
i.e. the dissipation of the Hamiltonian cannot happen instantaneously in time. The validity of \eqref{no instant dissipation idea} is rather general and it does not even require that much structure on the nonlinearity. To prove \eqref{no instant dissipation idea} one should just make an appropriate choice of the test function in the weak formulation (see the proof of Proposition \ref{P:no instant dissipation}). 

Note that, once \eqref{no instant dissipation idea} is established, one can split, for any $\delta>0$, the whole dissipation in $(0,T)$ into times $(0,\delta)$ and $(\delta,T)$. By first sending $\nu\rightarrow 0^+$, the dissipation in $(\delta,T)$ will vanish by \eqref{no diss for positive times idea}, while taking the subsequent limit $\delta\rightarrow 0^+$ shows that the dissipation in $(0,\delta)$ vanishes by \eqref{no instant dissipation idea}. We can therefore schematically summarize the three steps above as
\begin{align}
\text{Structure of the nonlinearity} \qquad &\Longrightarrow \qquad  \text{higher-order bound \eqref{higher order bound idea}}\\
    \text{higher-order bound \eqref{higher order bound idea}} \qquad &\Longrightarrow \qquad \text{no dissipation in } (\delta,T)\\
     \{\theta_0^\nu\}_\nu\subset \dot H^{-\alpha}(\T^2) \text{ strongly compact} \qquad &\Longrightarrow \qquad \text{dissipation arbitrarily small in } (0,\delta).
\end{align}
In Remark \ref{R: decessity of delta positive} we will see that the higher-order bound \eqref{higher order bound idea} does not generally hold for $\delta=0$. It is then necessary to treat the dissipation in $(0,\delta)$ and in $(\delta,T)$ separately. In other words, the steps $(2)$ and $(3)$ above cannot be merged in a single one.

\begin{remark}
    \label{R:main issues}
Although it certainly contains all the main ideas, the argument we have given above is only formal because of two main issues. First of all, since we are working with general values of $\alpha$ and $\gamma$, there might not be enough regularity to perform the computations, even if $\nu>0$. It is certainly not possible to assume the solution to be smooth since \eqref{g-SQG diss} might develop singularities in finite time. Therefore, to avoid the case in which the theorem applies to an empty set of solutions, it is necessary to work with weak solutions à la Leray (see Definition \ref{def.viscous-sqg-sol}), which always exist globally in time by classical compactness methods. It is here that a second issue arises. Due to the limited regularity of the  data, Leray solutions might not be unique, even if they satisfy the energy identity \eqref{dissip en bal}. It is then necessary to prove (see Proposition \ref{P:leray exist}) the existence of Leray solutions that satisfy \eqref{higher order bound idea} as an apriori bound, and work with them only. Indeed, as it should be clear from the above three steps, the argument applies once the higher-order bound holds. We  believe that, in the appropriate range of $\alpha$ and $\gamma$, there are Leray solutions in the sense of Definition \ref{def.viscous-sqg-sol} that do not satisfy the higher-order bound \eqref{higher order bound idea} and for which the conclusion of Theorem \ref{T:main diss} fails.
\end{remark}
The goal of the next sections is to make the whole argument rigorous.

\subsection{Weak solutions}\label{S:weak sol}

We want to define weak solutions to \eqref{g-SQG diss} and \eqref{g-SQG}. To do that, we start with the following formal manipulations
\begin{align}
 \int_{\mathbb{T}^2} \theta u\cdot \nabla \varphi \,dx  &= \int_{\mathbb{T}^2}  \mathcal{R}^\perp_\alpha\theta \cdot \nabla \varphi (-\Delta)^{\alpha}(-\Delta)^{-\alpha}\theta\, dx \\
 & = \int_{\mathbb{T}^2} \mathcal{R}^\perp_\alpha \theta \cdot [\nabla \varphi , (-\Delta)^{\alpha}](-\Delta)^{-\alpha}\theta \,dx +  \int_{\mathbb{T}^2} \mathcal{R}^\perp_\alpha \theta \cdot (-\Delta)^{\alpha} \left( \nabla \varphi (-\Delta)^{-\alpha}\theta \right) dx,
 \end{align}
 where $[\cdot,\cdot]$ is the usual commutator symbol. For convenience, we will denote it by 
 \begin{equation}
     \label{Talpha}
     T_\alpha [\varphi] := [\nabla \varphi , (-\Delta)^{\alpha}].
 \end{equation}
 We can perform further integrations by parts to rewrite the last term as
 \begin{align}
 \int_{\mathbb{T}^2} \mathcal{R}^\perp_\alpha \theta \cdot (-\Delta)^{\alpha} \left( \nabla \varphi (-\Delta)^{-\alpha}\theta \right) \,dx &= \int_{\mathbb{T}^2}(-\Delta)^{-\alpha}\theta \nabla \varphi \cdot \nabla^\perp \theta \,dx\\
 & = -\int_{\mathbb{T}^2} \theta \nabla^\perp (-\Delta)^{-\alpha}\theta \cdot \nabla \varphi  \,dx\\
  & = - \int_{\mathbb{T}^2} \theta u \cdot \nabla \varphi  \,dx.
 \end{align} 
 Thus, we arrive at
 \begin{align}
 \int_{\mathbb{T}^2} \theta u\cdot\nabla \varphi\, dx &= \frac{1}{2} \int_{\mathbb{T}^2} \mathcal{R}^\perp_\alpha \theta \cdot  T_\alpha [\varphi](-\Delta)^{-\alpha}\theta\, dx\\
 &=:\frac12 \left\langle \mathcal R^\perp_\alpha \theta , T_\alpha [\varphi](-\Delta)^{-\alpha}\theta \right\rangle.\label{reassembling nonlinearity}
 \end{align}
For any $\varphi\in C^\infty(\T^2)$, the linear operator $ T_\alpha [\varphi]:\dot H^s(\T^2)\rightarrow \dot H^{s+1-2\alpha}(\T^2)$ is continuous. In particular
$$
\| T_\alpha [\varphi] (-\Delta)^{-\alpha}\theta\|_{\dot H^{1-\alpha}} \leq C \|(-\Delta)^{-\alpha}\theta\|_{\dot H^{\alpha}}\leq C \|\theta\|_{\dot H^{-\alpha}}.
$$
Since also $\| R^\perp_\alpha \theta\|_{\dot H^{\alpha-1}}\leq C \|\theta\|_{\dot H^{-\alpha}}$, the coupling in \eqref{reassembling nonlinearity} is well-defined. We summarize these considerations in the following proposition which can be found in \cite{IM24}*{Theorem 4}.

\begin{proposition}[Continuity of nonlinearity]\label{P:nonlin continuous}
For any $\alpha\in (0,1]$ there exits a constant $C<\infty$ such that 
$$
\left|\left\langle \mathcal R^\perp_\alpha \theta , T_\alpha [\varphi](-\Delta)^{-\alpha}\theta \right\rangle\right|\leq C \|\varphi\|_{C^2} \|\theta\|^2_{\dot H^{-\alpha}},
$$
for all $\varphi\in C^\infty(\T^2)$ and all $\theta \in \dot H^{-\alpha}(\T^2)$. 
\end{proposition}
We have thus motivated the following definition.

\begin{definition}[Weak solutions to \eqref{g-SQG}]
    \label{def.sqg-solutions}
    Let $\alpha\in (0,1]$. Let $\theta_0\in \dot H^{-\alpha}(\T^2)$ be with zero average and $f\in L^1_{\rm loc}([0,\infty)\dot H^{-\alpha} (\T^2))$ be with zero spatial average for almost all times. We say that $\theta \in L^{2}_{\rm loc}([0,\infty), \dot H^{-\alpha}(\T^2))$, with $\langle \theta(t),1\rangle=0$ for almost every $t\geq 0$, is a weak solution to \eqref{g-SQG} if 
\begin{align}
    \int_0^\infty \left\langle \theta (t),\partial_t\varphi (t)\right\rangle dt &+\mez  \int_0^\infty \left\langle \mathcal R^\perp_\alpha \theta (t),T_\alpha[\varphi(t)](-\Delta)^{-\alpha}\theta(t) \right\rangle dt \\
    &= -\left\langle\theta_0,\varphi(0)\right\rangle - \int_0^\infty \langle f(t),\varphi(t)\rangle \,dt \label{eq.sqg-weak-sol}
\end{align}
for all $\varphi \in C^{\infty}_c([0,\infty)\times \T^2)$, where $\mathcal R^\perp_\alpha=\nabla^\perp(-\Delta)^{-\alpha}$ and $ T_\alpha [\varphi(t)] := [\nabla \varphi (t) , (-\Delta)^{\alpha}]$.
\end{definition}

\begin{remark}
    \label{R:nonlinearity continuous} 
 Proposition \ref{P:nonlin continuous} proves that, for any smooth test function $\varphi$, the nonlinear term in the distributional identity \eqref{eq.sqg-weak-sol} is a bilinear and continuous form on $L^2_{\rm loc}([0,\infty);\dot H^{-\alpha}(\T^2))$. 
\end{remark}

Weak solutions to \eqref{g-SQG diss} for $\nu>0$ are defined in a similar fashion. Because of the additional dissipative term, there is enough compactness to prove their existence for any $\dot H^{-\alpha}(\T^2)$ initial datum. Such solutions naturally come with an energy inequality and, therefore, we directly enforce it in the definition. Since \cite{L34}, weak solutions of a dissipative nonlinear PDE are commonly known as ``Leray solutions''.

\begin{definition}[Leray solutions to \eqref{g-SQG diss}]
    \label{def.viscous-sqg-sol}
    Let $\alpha\in (0,1]$, $\gamma\in (0,\infty)$ and $\nu>0$. Let $\theta^\nu_0\in \dot H^{-\alpha}(\T^2)$ be with zero average and $f^\nu\in L^1_{\rm loc}([0,\infty)\dot H^{-\alpha} (\T^2))$ be with zero spatial average for almost all times. We say that 
    $$
    \theta^\nu \in L^{\infty}_{\rm loc}([0,\infty), \dot H^{-\alpha}(\T^2))\cap L^{2}_{\rm loc}([0,\infty), \dot H^{\gamma-\alpha}(\T^2))\cap C^0_{\rm w}([0,\infty), \dot H^{-\alpha}(\T^2)),
    $$
    with $\langle \theta^\nu(t),1\rangle=0$ for all $t\geq 0$, is a weak solution to \eqref{g-SQG diss} if 
\begin{align}
    \int_0^\infty \left\langle \theta^\nu (t),\partial_t\varphi (t) - \nu (-\Delta)^\gamma \varphi (t)\right\rangle dt &+\mez  \int_0^\infty \left\langle \mathcal R^\perp_\alpha \theta^\nu (t),T_\alpha[\varphi(t)](-\Delta)^{-\alpha}\theta^\nu(t) \right\rangle dt  \\
    &= -\left\langle\theta^\nu_0,\varphi(0)\right\rangle - \int_0^\infty \langle f^\nu(t),\varphi(t)\rangle \,dt
\end{align}
for all $\varphi \in C^{\infty}_c([0,\infty)\times \T^2)$, where $\mathcal R^\perp_\alpha=\nabla^\perp(-\Delta)^{-\alpha}$ and $ T_\alpha [\varphi(t)] := [\nabla \varphi (t) , (-\Delta)^{\alpha}]$. In addition, the energy inequality
\begin{equation}
    \label{en ineq viscous gsqg}
  \frac12  \|\theta^\nu(t)\|^2_{\dot H^{-\alpha} } + \nu \int_0^t\|\theta^\nu(\tau)\|^2_{\dot H^{\gamma-\alpha}}\,d\tau \leq   \frac12  \|\theta^\nu_0\|^2_{\dot H^{-\alpha} } + \int_0^t\langle f^\nu (\tau),\theta^\nu(\tau)\rangle_{\dot H^{-\alpha}}\,d\tau
\end{equation}
holds for all $t\geq 0$, where $\langle \cdot,\cdot\rangle_{\dot H^{-\alpha}}$ denotes the $\dot H^{-\alpha}(\T^2)$ inner product.
\end{definition}

\begin{remark}\label{R:weak form}
Since $\theta^\nu\in C^0_{\rm w}([0,\infty), \dot H^{-\alpha}(\T^2))$, one can let the test function $\varphi(x,\tau)$ in the weak formulation converge to $\psi (x) \mathbbm{1}_{[0,t]}(\tau)$ and obtain
\begin{align}
   \left\langle\theta^\nu_0- \theta^\nu(t), \psi\right\rangle  &= \nu\int_0^t \left\langle \theta^\nu (\tau),  (-\Delta)^\gamma \psi\right\rangle d\tau -\mez  \int_0^t \left\langle \mathcal R^\perp_\alpha \theta^\nu (\tau),T_\alpha[\psi](-\Delta)^{-\alpha}\theta^\nu(\tau) \right\rangle d\tau  \\
    &\quad  - \int_0^t \langle f^\nu(\tau),\psi\rangle \, d\tau
\end{align}
for all $\psi\in C^\infty(\T^2)$ and all $t>0$.
\end{remark}

\begin{remark}
    \label{R:uniform hamilt bound}
  The energy inequality \eqref{en ineq viscous gsqg} yields to 
  \begin{align}
      \|\theta^\nu(t)\|^2_{\dot H^{-\alpha}} &\leq  \|\theta^\nu_0\|^2_{\dot H^{-\alpha}} + 2 \int_0^t\|f^\nu(\tau)\|_{\dot H^{-\alpha}} \|\theta^\nu(\tau)\|_{\dot H^{-\alpha}}\,d\tau \qquad \forall t>0.
  \end{align}
Therefore, by Gr\"onwall we deduce 
  $$
    \|\theta^\nu(t)\|_{\dot H^{-\alpha}} \leq \|\theta^\nu_0\|_{\dot H^{-\alpha}}  + \int_0^t \|f^\nu(\tau)\|_{\dot H^{-\alpha}}\,d\tau\qquad \forall t>0.
  $$
\end{remark}

Leray solutions in the sense of Definition \ref{def.viscous-sqg-sol} can be proved to exist by nowadays standard regularization and compactness methods. Indeed, by smoothness, the sequence of solutions to the regularized version of \eqref{g-SQG diss} satisfies \eqref{en ineq viscous gsqg} with the equality. Therefore, the sequence of approximate solutions stays bounded in $L^{\infty}_{\rm loc}([0,\infty), \dot H^{-\alpha}(\T^2))\cap L^{2}_{\rm loc}([0,\infty), \dot H^{\gamma-\alpha}(\T^2))$ (see also Remark \ref{R:uniform hamilt bound} above). Since $\gamma>0$, the Aubin--Lions lemma yields compactness in $L^2([0,\infty), \dot H^{-\alpha}(\T^2))$, which is enough to pass to the limit into the nonlinearity (see Remark \ref{R:nonlinearity continuous}). Then, the energy inequality \eqref{en ineq viscous gsqg} follows by the lower semicontinuity of the norms under weak convergence. The case $\alpha=\frac12$ was established in \cite{Mar08a} and it is not difficult to see that the method generalizes to any $\alpha\in (0,1]$. Note that such solutions are, in general, not unique \cites{AC23,MS06}. We refer to \cites{M06,J07,CCW01,CW99,DC12} and references therein for conditional uniqueness results. Because of the nonuniqueness issue, for us it is important to work with a subclass of Leray solutions. As it has already been explained in \cref{S:heuristic} and Remark \ref{R:main issues}, we will need Leray solutions that satisfy the additional higher-order bound \eqref{higher order bound idea}. Let us now prove that such solutions exist. As an outcome of the construction, we are also able to upgrade \eqref{en ineq viscous gsqg} to an equality. To the best of our knowledge this is also new.

\begin{proposition}[Leray solutions with higher-order bound]
    \label{P:leray exist}
    Let $\alpha\in (0,1]$, $\gamma\in (0,\infty)$ and $\nu\in (0,1)$. Let $\theta^\nu_0\in \dot H^{-\alpha}(\T^2)$ be with zero average and $f^\nu\in L^{\frac{\alpha+\gamma}{\gamma}}_{\rm loc}([0,\infty);\dot H^{-\alpha} (\T^2))$ be with zero spatial average for almost all times. If $\alpha>\gamma$ we make the stronger assumption that $f^\nu\in L^\infty_{\rm loc}([0,\infty);\dot H^{-\gamma} (\T^2))$. There exists a Leray solution to \eqref{g-SQG diss} in the sense of Definition \ref{def.viscous-sqg-sol} with the following additional properties
    \begin{itemize}
        \item[$(a)$] $ \theta^\nu \in C^0([0,\infty), \dot H^{-\alpha}(\T^2))\cap L^{2}_{\rm loc}([0,\infty), \dot H^{\gamma-\alpha}(\T^2))$;
        \item[$(b)$] for all $t\geq 0$ it holds 
        $$
        \frac12  \|\theta^\nu(t)\|^2_{\dot H^{-\alpha} } + \nu \int_{0}^t\|\theta^\nu(\tau)\|^2_{\dot H^{\gamma-\alpha}}\,d\tau =   \frac12  \|\theta^\nu_0\|^2_{\dot H^{-\alpha} } + \int_0^t\langle f^\nu (\tau),\theta^\nu(\tau)\rangle_{\dot H^{-\alpha}}\,d\tau;
        $$
        \item[$(c)$]  $ \theta^\nu \in L^2_{\rm loc} ((0,\infty);\dot H^\gamma (\T^2))$. More precisely, for any $0<\delta<T<\infty$ there exists a constant $M<\infty $ such that 
        \begin{equation}
            \label{hob}
             \nu^{\frac{\alpha+\gamma}{\gamma}} \int^T_\delta \|\theta^\nu(t)\|_{\dot H^{\gamma}}^2 \, dt\leq M.
        \end{equation}
        The constant $M$ depends\footnote{From the proof it will be clear that the dependence can be made explicit if needed.} only on $\delta,T, \|\theta^\nu_0\|_{\dot H^{-\alpha}}$, $\|f^\nu\|_{L^{\frac{\alpha+\gamma}{\gamma}}\dot H^{-\alpha}}$ if $\alpha\leq \gamma$ and on $\|f^\nu\|_{L^\infty\dot H^{-\gamma}}$ if $\alpha>\gamma$, but it is otherwise independent on $\nu\in (0,1)$.
    \end{itemize}
\end{proposition}

\begin{proof}

    The proof is a suitable adaptation of \cite{DLP25}*{Proposition 2.7}. The presence of the external force adds some technicalities in establishing the higher-order bound \eqref{hob}. We break the proof down into steps. Since here the viscosity parameter $\nu\in (0,1)$ will be fixed, we will remove it from the superscripts and simply denote a solution to \eqref{g-SQG diss} by $\theta$.

    \underline{\textsc{Step 1}}: Apriori estimates for smooth solutions.

    We start by proving several apriori bounds for smooth solutions to \eqref{g-SQG diss}. Recall the notation of the inner products from \cref{S:notation}. We will only make use of the following two cancellations of the nonlinearity
    \begin{equation}
        \label{cancellation hamiltonian}
        \langle u\cdot \nabla \theta,\theta\rangle_{\dot H^{-\alpha}}=\int_{\T^2}  u\cdot \nabla \theta (-\Delta)^{-\alpha}\theta\,dx = - \int_{\T^2}  \theta  u\cdot \nabla (-\Delta)^{-\alpha}\theta\,dx =0
    \end{equation}
and 
  \begin{equation}
        \label{cancellation L2}
        \langle u\cdot \nabla \theta,\theta\rangle=\int_{\T^2}  \theta u\cdot \nabla \theta \,dx =  \int_{\T^2}    u\cdot \nabla \frac{|\theta|^2}{2}\,dx =0.
    \end{equation}
The first one comes from $u=\nabla^\perp (-\Delta)^{-\alpha}\theta$, while the second follows from $\div u=0$. By \eqref{cancellation hamiltonian} we deduce 
\begin{equation}
    \label{ham cons in proof}
    \frac12  \|\theta(t)\|^2_{\dot H^{-\alpha} } + \nu \int_{s}^t\|\theta(\tau)\|^2_{\dot H^{\gamma-\alpha}}\,d\tau =   \frac12  \|\theta (s)\|^2_{\dot H^{-\alpha} } + \int_s^t\langle f (\tau),\theta(\tau)\rangle_{\dot H^{-\alpha}}\,d\tau
\end{equation}
for all $0\leq s<t$, while from \eqref{cancellation L2} we get 
\begin{equation}
    \label{L2 cons in proof}
    \frac12  \|\theta(t)\|^2_{L^2 } + \nu \int_{s}^t\|\theta(\tau)\|^2_{\dot H^{\gamma}}\,d\tau =   \frac12  \|\theta (s)\|^2_{L^2 } + \int_s^t\langle f (\tau),\theta(\tau)\rangle\,d\tau
\end{equation}
for all $0\leq s<t$. By \eqref{ham cons in proof} computed at $s=0$ it follows 
\begin{align}
    \|\theta(t)\|^2_{\dot H^{-\alpha}}\leq  \|\theta_0\|^2_{\dot H^{-\alpha}}+ 2\int_0^t  \|f(\tau)\|_{\dot H^{-\alpha}}\|\theta(\tau)\|_{\dot H^{-\alpha}}\,d\tau.
\end{align}
The Gr\"onwall inequality then yields to the first apriori bound 
\begin{equation}
    \label{1 apriori bound}
     \|\theta(t)\|_{\dot H^{-\alpha}}\leq \|\theta_0\|_{\dot H^{-\alpha}}+ \int_0^t  \|f(\tau)\|_{\dot H^{-\alpha}}\,d\tau \qquad \forall t>0.
\end{equation}

By plugging \eqref{1 apriori bound} into \eqref{ham cons in proof} we also deduce
\begin{align}
    \nu \int_{0}^t\|\theta(\tau)\|^2_{\dot H^{\gamma-\alpha}}\,d\tau\leq \frac12 \|\theta_0\|^2_{\dot H^{-\alpha}}+ \left(\|\theta_0\|_{\dot H^{-\alpha}} + \int_0^t  \|f(\tau)\|_{\dot H^{-\alpha}}\,d\tau\right) \int_0^t  \|f(\tau)\|_{\dot H^{-\alpha}}\,d\tau \label{2 apriori bound}
\end{align}
for all $t>0$, that is our second apriori bound.  Note that \eqref{1 apriori bound} and \eqref{2 apriori bound} only require $f\in L^1
_{\rm loc}([0,\infty);\dot H^{-\alpha}(\T^2))$. To derive the higher-order bound \eqref{hob} that is key for Theorem \ref{T:main diss}, we will need the extra assumptions on $f$. Let us start with the easier case 
\begin{equation}
    \label{subcrit}
    \alpha\leq \gamma.
\end{equation}
By interpolation and Young's inequality we have
\begin{align}
    \|\theta(\tau)\|^2_{L^2}\leq C   \|\theta(\tau)\|^{2\frac{\gamma-\alpha}{\gamma}}_{\dot H^{-\alpha}}\|\theta(\tau)\|^{2\frac{\alpha}{\gamma}}_{\dot H^{\gamma-\alpha}}\leq C\left(\nu^{-\frac{\alpha}{\gamma}}\|\theta(\tau)\|^{2}_{\dot H^{-\alpha}} + \nu^{\frac{\gamma-\alpha}{\gamma}}\|\theta(\tau)\|^{2}_{\dot H^{\gamma-\alpha}} \right).\label{young split}
\end{align}
Therefore, by \eqref{1 apriori bound} and \eqref{2 apriori bound} we deduce 
\begin{equation}
    \label{integrated L2}
    \int_0^t \|\theta(\tau)\|^2_{L^2}\,d\tau\leq C(t) \nu^{-\frac{\alpha}{\gamma}}\qquad \forall t>0,
\end{equation}
for a constant $C(t)<\infty$ for all $t>0$, monotone increasing in $t$, that only depends on $\|\theta_0\|_{\dot H^{-\alpha}}$ and $\int_0^t \|f(\tau)\|_{\dot H^{-\alpha}}\,d\tau$, but it is otherwise independent on $\nu>0$. We now want to use \eqref{integrated L2} into \eqref{L2 cons in proof} to achieve \eqref{hob}. Note that 
\begin{equation}\label{fubini}
\int_0^t\int_s^t\|\theta(\tau)\|^2_{\dot H^\gamma}\,d\tau ds= \int_0^t\tau \|\theta(\tau)\|^2_{\dot H^\gamma}\,d\tau
\end{equation}
by Fubini. Therefore, integrating \eqref{L2 cons in proof} in $s\in (0,t)$ and also using \eqref{integrated L2}, we deduce 
\begin{align}
    \frac{t}{2} \|\theta(t)\|^2_{L^2}+\nu\int_0^t\tau \|\theta(\tau)\|^2_{\dot H^\gamma}\,d\tau&\leq \frac12  \int_0^t \|\theta(s)\|^2_{L^2}\,ds\\
    &\quad +\int_0^t\int_s^t \|f(\tau)\|_{\dot H^{-\alpha}}\|\theta(\tau)\|_{\dot H^{\alpha}}\,d\tau ds\\
    &\leq C(t) \nu^{-\frac{\alpha}{\gamma}} +\int_0^t\int_s^t \|f(\tau)\|_{\dot H^{-\alpha}}\|\theta(\tau)\|_{\dot H^{\alpha}}\,d\tau ds.\label{estimate not generalized}
\end{align}
We are left to handle the very last term. Since we are working in the regime \eqref{subcrit}, we can interpolate
\begin{equation}\label{interp H alpha and H gamma}
\|\theta(\tau)\|_{\dot H^{\alpha}}\leq C \|\theta(\tau)\|_{\dot H^{-\alpha}}^{\frac{\gamma-\alpha}{\alpha+\gamma}}\|\theta(\tau)\|_{\dot H^{\gamma}}^{\frac{2\alpha}{\alpha+\gamma}}.
\end{equation}
Together with \eqref{1 apriori bound}, this yields to 
\begin{equation}
    \label{new interp}
 \int_0^t\int_s^t \|f(\tau)\|_{\dot H^{-\alpha}}\|\theta(\tau)\|_{\dot H^{\alpha}}\,d\tau ds\leq \tilde C(t) \int_0^t\int_s^t \|f(\tau)\|_{\dot H^{-\alpha}}\|\theta(\tau)\|^{\frac{2\alpha}{\alpha+\gamma}}_{\dot H^{\gamma}}\,d\tau ds,
\end{equation}
for a constant $\tilde C(t)<\infty$ for all $t>0$, monotone increasing in $t$. By Young's inequality 
$$
\|f(\tau)\|_{\dot H^{-\alpha}}\|\theta(\tau)\|^{\frac{2\alpha}{\alpha+\gamma}}_{\dot H^{\gamma}}\leq C(t) \nu^{-\frac{\alpha}{\gamma}} \|f(\tau)\|^{\frac{\alpha+\gamma}{\gamma}}_{\dot H^{-\alpha}} + \frac{\nu}{2\tilde C (t)} \|\theta(\tau)\|^2_{\dot H^\gamma}.
$$
Therefore, \eqref{new interp} becomes 
$$
\int_0^t\int_s^t \|f(\tau)\|_{\dot H^{-\alpha}}\|\theta(\tau)\|_{\dot H^{\alpha}}\,d\tau ds\leq C(t) \nu^{-\frac{\alpha}{\gamma}} \int_0^t \|f(\tau)\|^{\frac{\alpha+\gamma}{\gamma}}_{\dot H^{-\alpha}}\,d\tau + \frac{\nu}{2} \int_0^t \tau \|\theta(\tau)\|^2_{\dot H^\gamma}\,d\tau,
$$
where we have used again \eqref{fubini}. By plugging this last estimate into \eqref{estimate not generalized} we achieve 
\begin{align}
    \frac{t}{2} \|\theta(t)\|^2_{L^2}+\nu\int_0^t\tau \|\theta(\tau)\|^2_{\dot H^\gamma}\,d\tau \leq C(t) \nu^{-\frac{\alpha}{\gamma}} +\frac{\nu}{2}\int_0^t\tau \|\theta(\tau)\|^2_{\dot H^{\gamma}}\,d\tau.
\end{align}
Note that the last term can be absorbed in the left-hand-side. We have finally achieved, in the case $\alpha\leq \gamma$, the last two apriori estimates 
\begin{equation}
    \label{3 apriori bound}
    t\nu^{\frac{\alpha}{\gamma}}  \|\theta(t)\|^2_{L^2}\leq C(t)\qquad \forall t>0
\end{equation}
and 
\begin{equation}
    \label{4 apriori bound}
    \nu^{\frac{\alpha+\gamma}{\gamma}}  \int_0^t\tau \|\theta(\tau)\|^2_{\dot H^\gamma}\,d\tau\leq C(t)\qquad \forall t>0,
\end{equation}
for a constant $C(t)<\infty$, monotone increasing in $t$, depending only on $\|\theta_0\|_{\dot H^{-\alpha}}$, $ \int_0^t \|f(\tau)\|^{\frac{\alpha+\gamma}{\gamma}}_{\dot H^{-\alpha}}\,d\tau$ and $t$, but otherwise independent on $\nu\in(0,1)$.

We now have to handle the case 
\begin{equation}
    \label{supercrit}
    \alpha>\gamma.
\end{equation}
Note that in this case the above proof breaks down since we cannot use \eqref{young split} and \eqref{interp H alpha and H gamma} anymore. We now assume that $f\in L^\infty_{\rm loc}([0,\infty);\dot H^{-\gamma} (\T^2))$. By \eqref{L2 cons in proof} we have 
$$
\frac{d}{dt}\|\theta(t)\|^2_{L^2}=-2\nu \|\theta(t)\|^2_{\dot H^\gamma} +2\langle f(t),\theta(t)\rangle\qquad \forall t>0.
$$
We bound the term involving the force as 
$$
2\left|\langle f(t),\theta(t)\rangle\right|\leq 2 \|f(t)\|_{\dot H^{-\gamma}}\|\theta(t)\|_{\dot H^{\gamma}}\leq \nu^{-1} \|f(t)\|_{\dot H^{-\gamma}}^2 +\nu \|\theta(t)\|^2_{\dot H^{\gamma}},
$$
from which 
\begin{equation}
    \label{preparing ode}
    \frac{d}{dt}\|\theta(t)\|^2_{L^2}\leq -\nu \|\theta(t)\|^2_{\dot H^{\gamma}} +\nu^{-1} \|f(t)\|_{\dot H^{-\gamma}}^2\qquad \forall t>0.
\end{equation}
Interpolation together with \eqref{1 apriori bound} yields to
\begin{equation}
    \label{interp and ham bound}
    \|\theta(t)\|_{L^2}\leq C \|\theta(t)\|^{\frac{\gamma}{\alpha+\gamma}}_{\dot H^{-\alpha}} \|\theta(t)\|^{\frac{\alpha}{\alpha+\gamma}}_{\dot H^{\gamma}}\leq C(t) \|\theta(t)\|^{\frac{\alpha}{\alpha+\gamma}}_{\dot H^{\gamma}},
\end{equation}
for a constant $C(t)<\infty$ for all $t>0$, monotone increasing in $t$, depending only on $\|\theta_0\|_{\dot H^{-\alpha}}$ and $\int_0^t \|f(\tau)\|_{\dot H^{-\alpha}}\,d\tau$. Let us now fix an arbitrarily large time $T<\infty$. Since $\dot H^{-\gamma}(\T^2)\subset \dot H^{-\alpha}(\T^2)$ because of \eqref{supercrit}, by \eqref{interp and ham bound} we deduce 
$$
 \|\theta(t)\|_{L^2}\leq C \|\theta(t)\|^{\frac{\alpha}{\alpha+\gamma}}_{\dot H^{\gamma}}\qquad \forall t\in (0,T),
$$
for a constant $C<\infty$ that depends only on $\|\theta_0\|_{\dot H^{-\alpha}}$, $ \sup_{t\in (0,T)} \|f(t)\|_{\dot H^{-\gamma}}$ and $T$. Then, using again the assumption $f\in L^\infty_{\rm loc}([0,\infty);\dot H^{-\gamma} (\T^2))$, from \eqref{preparing ode} we derive 
\begin{equation}
    \label{ode to study}
     \frac{d}{dt}\|\theta(t)\|^2_{L^2}\leq -\nu C_1\|\theta(t)\|^{2\frac{\alpha+\gamma}{\alpha}}_{L^2} +\nu^{-1} C_2\qquad \forall t\in (0,T),
\end{equation}
with finite constants $C_1,C_2$ depending only on  $\|\theta_0\|_{\dot H^{-\alpha}}$, $ \sup_{t\in (0,T)} \|f(t)\|_{\dot H^{-\gamma}}$ and $T$. We now want to study the differential inequality \eqref{ode to study}. We claim that 
\begin{equation}
    \label{claim on ode}
    \|\theta(t)\|^2_{L^2}\leq\frac{C}{(\nu t)^\frac{\alpha}{\gamma}}\qquad \forall t\in (0,T),
\end{equation}
for a constant $C<\infty$ depending only on $C_1,C_2$ and $T$.

To prove the claim, we note that if there exists  $t_0\in [0,T)$ such that 
\begin{equation}\label{t0 below treshold}
\|\theta(t_0)\|^2_{L^2}< \left(\frac{2 C_2}{C_1 \nu^2}\right)^\frac{\alpha}{\alpha+\gamma}
\end{equation}
then necessarily
\begin{equation}\label{trapped for all t}
\|\theta(t)\|^2_{L^2}\leq \left(\frac{2 C_2}{C_1 \nu^2}\right)^\frac{\alpha}{\alpha+\gamma}\quad \forall t\in (t_0,T).
\end{equation}
This is quite immediate to see since saturating\footnote{Meaning that \eqref{trapped for all t} becomes an equality.} \eqref{trapped for all t} for some $t>t_0$ would make the right-hand-side of \eqref{ode to study} strictly negative so that $\|\theta(t)\|^2_{L^2}$ must decrease. In other words, the right-hand-side in \eqref{t0 below treshold} is a ``trapping-threshold'' for the differential inequality \eqref{ode to study}. Therefore, if \eqref{t0 below treshold} holds at $t_0=0$, then by \eqref{supercrit} and for $\nu\in (0,1)$ we deduce
$$
\|\theta(t)\|^2_{L^2}\leq  \frac{C}{\nu^{\frac{2\alpha}{\alpha+\gamma}}}\leq \frac{C}{\nu^{\frac{\alpha}{\gamma}}}\leq \frac{C T^\frac{\alpha}{\gamma}}{(\nu t)^{\frac{\alpha}{\gamma}}} \qquad \forall t\in (0,T),
$$
that is \eqref{claim on ode} holds. So assume 
\begin{equation}
    \label{initial condition}
\|\theta_0\|^2_{L^2}\geq \left(\frac{2 C_2}{C_1 \nu^2}\right)^\frac{\alpha}{\alpha+\gamma}.
\end{equation}
In this case, the left-hand-side in \eqref{ode to study} is strictly negative at $t=0$. It follows that the function $\|\theta(t)\|^2_{L^2}$ starts decreasing. Define 
\begin{equation}
    t_*:=\inf \left\{t\in (0,T)\,:\,  \|\theta(t)\|^2_{L^2}< \left(\frac{2 C_2}{C_1 \nu^2}\right)^\frac{\alpha}{\alpha+\gamma}\right\},
\end{equation}
by implicitly setting $t_*=T$ if the required condition does not hold for any $t\in (0,T)$.
If $t_*=0$\footnote{This happens if and only if \eqref{initial condition} is an equality.} then we are done by the argument above. So assume $t_*\in (0,T]$. Note that, in the case $t_*\in (0,T)$, from $t_*$ on the claim \eqref{claim on ode} holds by the same argument we have already given. We are then left to study the behavior for $t\in (0,t_*)$. By the definition of $t_*$ it follows 
$$
\|\theta(t)\|^2_{L^2}\geq \left(\frac{2 C_2}{C_1 \nu^2}\right)^\frac{\alpha}{\alpha+\gamma} \qquad \forall t\in (0,t_*).
$$
With this restriction \eqref{ode to study} becomes 
\begin{equation}
    \label{ode autonomous}
     \frac{d}{dt}\|\theta(t)\|^2_{L^2}\leq -\nu \frac{C_1}{2}\|\theta(t)\|^{2\frac{\alpha+\gamma}{\alpha}}_{L^2} \qquad \forall t\in (0,t_*).
\end{equation}
Then, by the comparison principle,  $\|\theta(t)\|^2_{L^2}$ must stay below the solution of the corresponding ODE, that is
$$
\|\theta(t)\|^{2}_{L^2}\leq \frac{1}{\left(\nu t C + \|\theta_0\|^{-2\frac{\gamma}{\alpha}}_{L^2}\right)^\frac{\alpha}{\gamma}}\leq \frac{C}{(\nu t)^\frac{\alpha}{\gamma}}\qquad \forall t\in (0,t_*),
$$
for a constant $C<\infty$ depending only on $C_1$, $\alpha$ and $\gamma$. The proof of the claim \eqref{claim on ode} is concluded.

By integrating \eqref{preparing ode} and using \eqref{claim on ode} we achieve
\begin{align}
    \nu\int_s^t\|\theta(\tau)\|^2_{\dot H^\gamma}\,d\tau \leq \|\theta(s)\|^2_{L^2} +\nu^{-1} \int_s^t \|f(\tau)\|^2_{\dot H^{-\gamma}}\,d\tau\leq C\left(\frac{1}{(\nu s)^{\frac{\alpha}{\gamma}}} +\frac{1}{\nu}\right)
\end{align}
for all $0<s<t<T$. Since we are working under the assumption \eqref{supercrit} and $\nu\in (0,1)$, the latter inequality yields to the last desired higher-order bound that reads as follows. For any $0<s<T<\infty$, there exists a constant $C<\infty$ such that 
\begin{equation}
    \label{5 apriori bound}
    \nu^{\frac{\alpha+\gamma}{\gamma}}\int_s^t\|\theta(\tau)\|^2_{\dot H^\gamma}\,d\tau\leq C\qquad \forall t\in (s,T).
\end{equation}
The constant $C$ depends only on $s$, $T$, $\|\theta_0\|_{\dot H^{-\alpha}}$ and $\sup_{t\in (0,T)} \|f(t)\|_{\dot H^{-\gamma}}$, but it is otherwise independent on $\nu \in(0,1)$.

    \underline{\textsc{Step 2}}: Regularizing the PDE.

We now want to regularize \eqref{g-SQG diss} in order to obtain a smooth sequence $\{\theta^\eps\}_\eps$ of approximate solutions. In doing that, it is important to use a regularization scheme that preserves the structure of the nonlinearity so that $\{\theta^\eps\}_\eps$ enjoys all the apriori bounds proved in Step 1. More precisely, we want to preserve the two cancellations \eqref{cancellation hamiltonian} and \eqref{cancellation L2}. 

For any $\eps>0$, denote by $J_\eps$ the space mollification operator on $\R^2$ with a radially symmetric kernel and by $\tilde J_\eps$ the one in space-time. Denote by $\theta^\eps$ a solution to 
\begin{equation}\label{g-SQG diss mollified}
\left\{\begin{array}{l}
\partial_t\theta^\eps + J_\eps \left(J_\eps u^\eps \cdot\nabla J_\eps \theta^\eps\right) + \nu (-\Delta)^\gamma \theta^\eps = \tilde J_\eps f\\ 
u^\eps = \mathcal R^\perp_\alpha \theta^\eps\\
\theta^\eps(\cdot,0)=J_\eps\theta_0.
\end{array}\right. 
\end{equation}
By classical methods (see for instance \cite{BV22}*{Chapter 4} and \cite{majda2002vorticity}*{Chapter 3}) the system \eqref{g-SQG diss mollified} admits, for any $\eps>0$, a unique solution $\theta^\eps$, which is moreover smooth in space-time.  It is immediate to check that the regularized nonlinearity in \eqref{g-SQG diss mollified} preserves the two cancellations \eqref{cancellation hamiltonian} and \eqref{cancellation L2}. It follows that $\theta^\eps$ enjoys all the apriori bounds proved in Step 1. In particular, since $\theta_0\in \dot H^{-\alpha}(\T^2)$ and $f\in L^1_{\rm loc}([0,\infty); \dot H^{-\alpha}(\T^2))$, we deduce 
\begin{align}\label{unif est 1}
    \{\theta^\eps\}_\eps \text{ is bounded in } L^\infty_{\rm loc}([0,\infty); \dot H^{-\alpha}(\T^2)) \text{ by \eqref{1 apriori bound}}
\end{align}
and 
\begin{align}\label{unif est 2}
    \{\theta^\eps\}_\eps \text{ is bounded in } L^2_{\rm loc}([0,\infty); \dot H^{\gamma-\alpha}(\T^2)) \text{ by \eqref{2 apriori bound}}.
\end{align}
Moreover, since in addition $f\in L^{\frac{\alpha+\gamma}{\gamma}}_{\rm loc}([0,\infty); \dot H^{-\alpha}(\T^2))$ if $\alpha\leq \gamma$ and $f\in L^\infty_{\rm loc}([0,\infty); \dot H^{-\gamma}(\T^2))$ in the case $\alpha >\gamma$, we also have 
\begin{align}\label{unif est 3}
    \{\theta^\eps\}_\eps \text{ is bounded in } L^\infty_{\rm loc}((0,\infty); L^2(\T^2)) \text{ by \eqref{3 apriori bound} and \eqref{claim on ode}}
\end{align}
and 
\begin{align}\label{unif est 4}
    \{\theta^\eps\}_\eps \text{ is bounded in } L^2_{\rm loc}((0,\infty); \dot H^{\gamma}(\T^2)) \text{ by \eqref{4 apriori bound} and \eqref{5 apriori bound}}.
\end{align}
Also,  by smoothness, for any $\eps>0$ we have
\begin{equation}
    \label{ham bal eps}
     \frac12  \|\theta^\eps(t)\|^2_{\dot H^{-\alpha} } + \nu \int_{s}^t\|\theta^\eps(\tau)\|^2_{\dot H^{\gamma-\alpha}}\,d\tau =   \frac12  \|\theta^\eps (s)\|^2_{\dot H^{-\alpha} } + \int_s^t\langle f^\eps (\tau),\theta^\eps(\tau)\rangle_{\dot H^{-\alpha}}\,d\tau
\end{equation}
for all $0\leq s<t<\infty$.

   \underline{\textsc{Step 3}}: Compactness and conclusion.

    In Step 2 we have obtained a sequence $\{\theta^\eps\}_\eps$ of smooth solutions to \eqref{g-SQG diss mollified} enjoying \eqref{unif est 1}, \eqref{unif est 2}, \eqref{unif est 3} and \eqref{unif est 4}. By weak* compactness and the Aubin--Lions lemma we find a (nonrelabelled) subsequence and a function $\theta$ such that 
    \begin{align}
        \theta^\eps \overset{*}{\rightharpoonup}\theta \qquad &\text{ in } L^\infty_{\rm loc}([0,\infty);\dot H^{-\alpha}(\T^2)),\label{weak conv ham}\\
           \theta^\eps \overset{*}{\rightharpoonup}\theta \qquad &\text{ in } L^2_{\rm loc}([0,\infty);\dot H^{\gamma-\alpha}(\T^2)),\label{weak conv diss}\\
           \theta^\eps \rightarrow\theta \qquad &\text{ in } L^2_{\rm loc}([0,\infty);\dot H^{-\alpha}(\T^2)),\label{strong L2 conv ham}\\
           \theta^\eps \rightarrow\theta \qquad &\text{ in } L^2_{\rm loc}((0,\infty);\dot H^{\gamma-\alpha}(\T^2))\label{strong L2 conv diss}
    \end{align}
and
\begin{equation}\label{unif conv ham}
    \theta^\eps \rightarrow\theta \qquad \text{ in } C^0_{\rm loc}((0,\infty);\dot H^{-\alpha}(\T^2)).
\end{equation}
Note that $\theta$ must satisfy  \eqref{hob} by lower semicontinuity. This proves $(c)$ in the statement of Proposition \ref{P:leray exist}. By \eqref{strong L2 conv ham} together with Remark \ref{R:nonlinearity continuous} we deduce that $\theta$ satisfies the weak formulation from Definition \ref{def.viscous-sqg-sol}, i.e. it is a weak solution to \eqref{g-SQG diss}. By \eqref{strong L2 conv diss} and \eqref{unif conv ham} we can let $\eps\rightarrow 0^+$ in \eqref{ham bal eps} as soon as $s>0$ to get
\begin{equation}
    \label{exact ham bal positive times}
    \frac12  \|\theta(t)\|^2_{\dot H^{-\alpha} } + \nu \int_{s}^t\|\theta(\tau)\|^2_{\dot H^{\gamma-\alpha}}\,d\tau =   \frac12  \|\theta (s)\|^2_{\dot H^{-\alpha} } + \int_s^t\langle f (\tau),\theta(\tau)\rangle_{\dot H^{-\alpha}}\,d\tau
\end{equation}
for all $0<s<t<\infty$. We are only left to show 
\begin{equation}
\label{strong cont at time 0}
\lim_{t\rightarrow 0^+}\|\theta(t)\|_{\dot H^{-\alpha} }=\|\theta_0\|_{\dot H^{-\alpha} }.
\end{equation}
Indeed, once this is established, the strong continuity up to the initial time claimed in $(a)$ holds. This also allows to let $s\rightarrow 0^+$ in \eqref{exact ham bal positive times} and deduce the validity of $(b)$, concluding the proof of the proposition. 

By \eqref{weak conv ham} we have $\theta \in L^\infty_{\rm loc}([0,\infty); \dot H^{-\alpha}(\T^2))$. Then
\begin{equation}
    \label{leq}
    \|\theta_0\|_{\dot H^{-\alpha} }\leq \liminf_{t\rightarrow 0^+}\|\theta(t)\|_{\dot H^{-\alpha} }
\end{equation}
follows by the lower semicontinuity of the norm under weak convergence\footnote{Note that $\theta(t)\overset{*}{\rightharpoonup } \theta_0$ in $\dot H^{-\alpha}(\T^2)$ as $t\rightarrow 0^+$ by the weak formulation.}. On the other hand, by \eqref{weak conv ham}, \eqref{weak conv diss} and \eqref{unif conv ham}, by letting $\eps\rightarrow 0^+$ in \eqref{ham bal eps} for $s=0$ we obtain 
$$
\frac12  \|\theta(t)\|^2_{\dot H^{-\alpha} } + \nu \int_{0}^t\|\theta(\tau)\|^2_{\dot H^{\gamma-\alpha}}\,d\tau \leq   \frac12  \|\theta_0\|^2_{\dot H^{-\alpha} } + \int_0^t\langle f (\tau),\theta(\tau)\rangle_{\dot H^{-\alpha}}\,d\tau\qquad \forall t>0.
$$
Since the function $\tau\mapsto \langle f (\tau),\theta(\tau)\rangle_{\dot H^{-\alpha}}$ belongs to $L^1_{\rm loc}([0,\infty))$, the latter inequality yields to 
\begin{equation}
    \label{geq}
     \limsup_{t\rightarrow 0^+}\|\theta(t)\|_{\dot H^{-\alpha} }\leq \|\theta_0\|_{\dot H^{-\alpha} }.
\end{equation}
The two inequalities \eqref{leq} and \eqref{geq} together prove \eqref{strong cont at time 0}, concluding the proof.
\end{proof}

\begin{remark}
    As it is clear from the proof, the solution constructed in Proposition \ref{P:leray exist} also enjoys
    $$
    \sup_{t\in [\delta,T]} \|\theta^\nu(t)\|^2_{L^2}<\infty \qquad \text{for any } 0<\delta<T<\infty.
    $$
    We have decided not to include this additional estimate in the statement of Proposition \ref{P:leray exist} since it is not necessary to prove Theorem \ref{T:main diss}.
\end{remark}

\subsection{Compactness implies no dissipation}
In this section we prove Theorem \ref{T:main diss}. To do that, we need the following proposition excluding instantaneous dissipation. We emphasize that this does not require any of the additional properties proved in Proposition \ref{P:leray exist} and it works for weak solutions in the sense of Definition \ref{def.viscous-sqg-sol}.

\begin{proposition}[Initial compactness \& dissipation]
    \label{P:no instant dissipation}
     Let $\alpha\in (0,1]$, $\gamma\in (0,\infty)$ and $\nu \in (0,1)$. Let $\theta^\nu_0\in \dot H^{-\alpha}(\T^2)$ be with zero average and $f^\nu\in L^1_{\rm loc}([0,\infty);\dot H^{-\alpha} (\T^2))$ be with zero spatial average for almost all times. Assume that $\{f^\nu\}_\nu\subset L^1([0,\delta_0];\dot H^{-\alpha} (\T^2))$ is bounded for some $\delta_0>0$ and 
     \begin{equation}
         \label{weak comp of force}
         \lim_{\delta\rightarrow 0^+} \sup_{\nu \in (0,1)} \nu\int_0^\delta \| f^\nu (t) \|_{\dot H^{- \alpha}} \, dt =0.
     \end{equation}
      Let $\theta^\nu$ be a weak solution to \eqref{g-SQG diss} in the sense of Definition \ref{def.viscous-sqg-sol}. Then
     \begin{equation}
    \label{no instant dissipation}
     \{\theta_0^\nu\}_\nu \subset \dot H^{-\alpha}(\T^2) \text{ strongly compact} \quad \Longrightarrow \quad \lim_{\delta\rightarrow 0^+} \sup_{\nu \in (0,1)} \nu\int_0^\delta \| \theta^\nu (t) \|_{\dot H^{\gamma - \alpha}}^{2} \, dt =0.
\end{equation}
\end{proposition}
The assumption \eqref{weak comp of force} prevents atomic concentrations at the initial time due to external forcing and, therefore, it is quite natural. It is satisfied, for instance, whenever $\{\|f^\nu (\cdot)\|_{\dot H^{-\alpha}}\}_\nu\subset L^1([0,\delta_0])$ is weakly compact for some $\delta_0>0$. The latter holds if $\{f^\nu \}_\nu$ is bounded in $L^p([0,\delta_0];\dot H^{-\alpha} (\T^2))$ for some $p>1$.

\begin{proof}
    Let $0<\delta_0<\delta$. By the energy inequality \eqref{en ineq viscous gsqg} we have 
    \begin{align}
        2\nu \int_0^\delta \|\theta^\nu(\tau)\|^2_{\dot H^{\gamma-\alpha}}\,d\tau &\leq  \|\theta^\nu_0\|^2_{\dot H^{-\alpha}} -  \|\theta^\nu(\delta)\|^2_{\dot H^{-\alpha}} + 2 \int_0^\delta \langle f^\nu (\tau),\theta^\nu(\tau)\rangle_{\dot H^{-\alpha}}\,d\tau\\
        &= \langle \theta^\nu_0 - \theta^\nu(\delta),\theta^\nu_0 + \theta^\nu(\delta)\rangle_{\dot H^{-\alpha}} + 2 \int_0^\delta \langle f^\nu (\tau),\theta^\nu(\tau)\rangle_{\dot H^{-\alpha}}\,d\tau\\
        &\leq  \|\theta^\nu_0 -\theta^\nu(\delta) \|_{\dot H^{-\alpha}}\|\theta^\nu_0 + \theta^\nu(\delta) \|_{\dot H^{-\alpha}}\\
        &\quad + 2 \int_0^\delta \|f^\nu (\tau)\|_{\dot H^{-\alpha}} \|\theta^\nu (\tau)\|_{\dot H^{-\alpha}}\,d\tau.\label{bound 1}
    \end{align}
    By Remark \ref{R:uniform hamilt bound}, since both $\{\theta_0^\nu\}_\nu\subset \dot H^{-\alpha}(\T^2)$ and $\{f^\nu\}_\nu \subset L^1([0,\delta_0];\dot H^{-\alpha}(\T^2))$ are bounded, we deduce
    \begin{equation}
        \label{bound 2}
         \sup_{t\in [0,\delta], \, \nu \in (0,1)}  \|\theta^\nu(t) \|_{\dot H^{-\alpha}}<\infty.
    \end{equation}
    Therefore, the previous estimate becomes 
    \begin{align}
         2\nu \int_0^\delta \|\theta^\nu(\tau)\|^2_{\dot H^{\gamma-\alpha}}\,d\tau &\leq C\left(  \|\theta^\nu_0 -\theta^\nu(\delta) \|_{\dot H^{-\alpha}} + \int_0^\delta \|f^\nu(\tau)\|_{\dot H^{-\alpha}}\,d\tau\right). \label{bound 3}
    \end{align}
    We are thus left to estimate $\|\theta^\nu_0 -\theta^\nu(\delta) \|_{\dot H^{-\alpha}}$. Let us expand it and bound it as 
    \begin{align}
        \|\theta^\nu_0 -\theta^\nu(\delta) \|^2_{\dot H^{-\alpha}}&= \|\theta^\nu(\delta) \|^2_{\dot H^{-\alpha}} - \|\theta^\nu_0 \|^2_{\dot H^{-\alpha}}+ 2 \langle \theta^\nu_0, \theta^\nu_0-\theta^\nu(\delta)\rangle_{\dot H^{-\alpha}}\\
        &\leq 2 \left( \int_0^\delta \langle f^\nu(\tau),\theta^\nu(\tau)\rangle_{\dot H^{-\alpha}} \,d\tau  + \langle \theta^\nu_0, \theta^\nu_0-\theta^\nu(\delta)\rangle_{\dot H^{-\alpha}}\right)\\
        &\leq C\left(\int_0^\delta \|f^\nu(\tau)\|_{\dot H^{-\alpha}}\,d\tau + \langle (\theta^\nu_0)_{\leq N}, \theta^\nu_0-\theta^\nu(\delta)\rangle_{\dot H^{-\alpha}} +\Phi(N)\right),\label{bound 4}
    \end{align}
    where we have used again \eqref{bound 2} and we have set 
    $$
    \Phi(N):=\sup_{\nu \in (0,1)} \| (\theta^\nu_0)_{> N}\|_{\dot H^{-\alpha}}.
    $$
    By Remark \ref{R:weak form} we can use $(\theta^\nu_0)_{\leq N}$ as a time-independent test function in the weak formulation (recall the identity \eqref{duality and inner product}) to get 
    \begin{align}
         \langle (\theta^\nu_0)_{\leq N}, \theta^\nu_0-\theta^\nu(\delta)\rangle_{\dot H^{-\alpha}} &= \left\langle |\nabla|^{-2\alpha}(\theta^\nu_0)_{\leq N}, \theta^\nu_0-\theta^\nu(\delta)\right\rangle\\
         &= \nu \int_0^\delta  \langle \theta^\nu(\tau) , (-\Delta)^\gamma |\nabla|^{-2\alpha}(\theta^\nu_0)_{\leq N}\rangle \,d\tau\\
         &\quad -\mez  \int_0^\delta \left\langle \mathcal R^\perp_\alpha \theta^\nu (\tau),T_\alpha\left[|\nabla|^{-2\alpha}(\theta^\nu_0)_{\leq N}\right](-\Delta)^{-\alpha}\theta^\nu(\tau) \right\rangle d\tau\\
         &\quad -\int_0^\delta \langle f^\nu(\tau) , |\nabla|^{-2\alpha}(\theta^\nu_0)_{\leq N}\rangle \,d\tau.
    \end{align}
    Therefore, since $\nu<1$, by Proposition \ref{P:nonlin continuous} together with \eqref{bound 2} we deduce 
    \begin{align}
        \left| \langle (\theta^\nu_0)_{\leq N}, \theta^\nu_0-\theta^\nu(\delta)\rangle_{\dot H^{-\alpha}}\right|&\leq \| (-\Delta)^\gamma (\theta^\nu_0)_{\leq N}\|_{\dot H^{-\alpha}}\int_0^\delta \| \theta^\nu (\tau)\|_{\dot H^{-\alpha}}\,d\tau\\
        &\quad +C \| |\nabla|^{-2\alpha}(\theta^\nu_0)_{\leq N}\|_{C^2} \int_0^\delta \| \theta^\nu (\tau)\|^2_{\dot H^{-\alpha}}\,d\tau\\
        &\quad + \|(\theta^\nu_0)_{\leq N}\|_{\dot H^{-\alpha}} \int_0^\delta \| f^\nu (\tau)\|_{\dot H^{-\alpha}}\,d\tau\\
        &\leq C\left(\delta \|  (\theta^\nu_0)_{\leq N}\|_{\dot H^{L-\alpha}} +\int_0^\delta \|f^\nu(\tau)\|_{\dot H^{-\alpha}}\,d\tau \right)  \\
        &\leq C\left(\delta N^L +\int_0^\delta \|f^\nu(\tau)\|_{\dot H^{-\alpha}}\,d\tau \right),
    \end{align}
    for a sufficiently large $L>1$ depending on $\gamma$. By plugging the latter estimate into \eqref{bound 4} we arrive at
    $$
         \|\theta^\nu_0 -\theta^\nu(\delta) \|^2_{\dot H^{-\alpha}}\leq C \left(\delta N^L +\int_0^\delta \|f^\nu(\tau)\|_{\dot H^{-\alpha}}\,d\tau  +\Phi(N)\right).
    $$
    Going back to \eqref{bound 3} we have proved 
    $$
    \nu \int_0^\delta \|\theta^\nu(\tau)\|^2_{\dot H^{\gamma-\alpha}}\,d\tau \leq C\left( \sqrt{\delta N^L +\int_0^\delta \|f^\nu(\tau)\|_{\dot H^{-\alpha}}\,d\tau  +\Phi(N)} +\int_0^\delta \|f^\nu(\tau)\|_{\dot H^{-\alpha}}\,d\tau \right),
    $$
    for a constant $C<\infty$ that does not depend on any of the parameters $\delta,N$ and $\nu$. Therefore, thanks to \eqref{weak comp of force}, we deduce
    $$
    \lim_{\delta\rightarrow 0^+}\sup_{\nu \in (0,1)}  \nu \int_0^\delta \|\theta^\nu(\tau)\|^2_{\dot H^{\gamma-\alpha}}\,d\tau \leq C\sqrt{\Phi(N)}\qquad \forall N>0,
    $$
    which concludes the proof since $\Phi(N)\rightarrow 0$ as $N\rightarrow \infty$ by the strong compactness of $\{\theta^\nu_0\}_\nu$ in $\dot H^{-\alpha}(\T^2)$.
\end{proof}

\begin{remark}
    \label{R:initial strong compactness}
    Without the strong compactness of the initial data Proposition \ref{P:no instant dissipation} fails. We give an example on the whole space $\R^2$. For a nontrivial radial and average-free $\theta_0\in C^\infty_c(\R^2)$, solve $\partial_t \theta =\Delta \theta$. Then 
    $$
    \theta^\nu(x,t):=\frac{1}{\nu^{\frac{1+\alpha}{\gamma}}}\theta\left(\frac{x}{\nu^{\frac{1}{\gamma}}},\frac{t}{\nu}\right)
    $$
    solves $\partial_t \theta^\nu =-\nu (-\Delta)^\gamma \theta^\nu$ with initial datum  
    $$
    \theta_0^\nu (x) :=\frac{1}{\nu^{\frac{1+\alpha}{\gamma}}}\theta_0\left(\frac{x}{\nu^{\frac{1}{\gamma}}}\right).
    $$
    Since $\theta^\nu$ is radially symmetric, the nonlinear term vanishes and $\theta^\nu$ is a solution to \eqref{g-SQG diss} with $f^\nu=0$. Moreover 
    $$
    \nu\int_0^\nu \|\theta^\nu(\tau)\|_{\dot H^{\gamma -\alpha}}^2\,d\tau= \int_0^1 \|\theta(\tau)\|^2_{\dot H^{\gamma -\alpha}}\,d\tau\qquad \text{for all } \nu>0.
    $$
    It follows that the dissipation of the Hamiltonian is of order $1$ in any time interval $(0,\delta)$. A direct inspection shows that indeed the sequence $\{\theta^\nu_0\}_\nu$ develops atomic concentration in $\dot H^{-\alpha}(\R^2)$ and therefore it is not strongly compact. For any $s\in \R$ we compute
\begin{equation}\label{compute norm rescaling}
    \int_0^T \|\theta^\nu(t)\|^2_{\dot H^s}\,dt =\nu^{1-\frac{2}{\gamma}(\alpha+s)}\int_0^{\frac{T}{\nu}}\|\theta(t)\|^2_{\dot H^s}\,dt.
\end{equation}
    Therefore $\theta^\nu\rightarrow 0$ in $L^2([0,T];\dot H^s(\R^2))$ if $s<-\alpha +\frac{\gamma}{2}$. Thus, the strong $L^2([0,T];\dot H^{-\alpha}(\R^2))$ compactness of the sequence of solutions is not enough to prevent the dissipation to concentrate at the initial time.
\end{remark}

\begin{remark}\label{R: decessity of delta positive}
    The sequence constructed in Remark \ref{R:initial strong compactness} also shows that, at the least on $\R^2$, the higher-order bound \eqref{hob} cannot generally hold for $\delta=0$. Indeed, by \eqref{compute norm rescaling} for $s=\gamma$ we get
    \begin{equation}\label{unbounded}
\nu^{\frac{\alpha+\gamma}{\gamma}}\int_0^T \|\theta^\nu(t)\|^2_{\dot H^\gamma}\,dt = \nu^{-\frac{\alpha}{\gamma}}\int_0^{\frac{T}{\nu}}\|\theta(t)\|^2_{\dot H^\gamma}\,dt,
        \end{equation}
    which blows up as $\nu\rightarrow 0^+$. Since there is a  bit of room to make \eqref{unbounded} unbounded, the example can be easily adjusted\footnote{Multiply $\theta^\nu$ by $\nu^\kappa$ for $ \kappa\in (0,\frac{\alpha}{2\gamma})$.} to also have $\theta^\nu_0\rightarrow 0$ in $\dot H^{-\alpha}(\R^2)$. The reason why \eqref{hob} is still valid for $\delta>0$ is due to the fact that 
    $$
    \int_{\frac{\delta}{\nu}}^{\frac{T}{\nu}} \|\theta(t)\|^2_{\dot H^\gamma}\,dt
    $$
    vanishes pretty fast as $\nu\rightarrow 0^+$ since $\theta$ is a solution of the fractional heat equation with a zero-average smooth initial datum.
\end{remark}

We are now ready to prove our first main theorem.
\begin{proof}[Proof of Theorem \ref{T:main diss}]
    Let $\eps>0$ and $\nu\in (0,1)$. Note that the assumptions we are making in $(ii)$ on the sequence of forcing are always sufficient to deduce the validity of \eqref{weak comp of force}. Then, since we are assuming $\{\theta^\nu_0\}_\nu\subset \dot H^{-\alpha}(\T^2)$ to be strongly compact, by Proposition \ref{P:no instant dissipation} we find $\delta>0$ such that 
    \begin{equation}
        \label{dissipation small for small times}
        \sup_{\nu \in (0,1)} \nu\int_0^\delta \| \theta^\nu (\tau) \|_{\dot H^{\gamma - \alpha}}^{2} \, d\tau<\eps.
    \end{equation}
    Therefore 
    \begin{equation}
        \label{split diss for short and positive times}
         \nu\int_0^T \| \theta^\nu (\tau) \|_{\dot H^{\gamma - \alpha}}^{2} \, d\tau<\eps +  \nu\int_\delta^T \| \theta^\nu (\tau) \|_{\dot H^{\gamma - \alpha}}^{2} \, d\tau.
    \end{equation}
    We now have to estimate the last term in the latter inequality. With the notation from \eqref{freq cuts}, for any $N>0$, we split it as 
    \begin{equation}\label{dissipation split low and high freq}
    \nu \int_\delta^T \| \theta^\nu (\tau) \|_{\dot H^{\gamma - \alpha}}^{2} \, d\tau=  \nu  \int_\delta^T \| \theta_{\leq N}^\nu (\tau) \|_{\dot H^{\gamma - \alpha}}^{2} \, d\tau+ \nu \int_\delta^T \| \theta_{>N}^\nu (\tau) \|_{\dot H^{\gamma - \alpha}}^{2} \, d\tau.
        \end{equation}
    By Remark \ref{R:uniform hamilt bound} we estimate the low-frequency term as
    \begin{equation}\label{est ugly 3}
     \int_\delta^T \| \theta_{\leq N}^\nu (\tau) \|_{\dot H^{\gamma - \alpha}}^{2} \, d\tau\leq N^{2\gamma} \int_\delta^T \| \theta^\nu (\tau) \|_{\dot H^{- \alpha}}^{2} \, d\tau\leq C  N^{2\gamma} \sup_{\nu \in (0,1)}\|\theta^\nu_0\|^2_{\dot H^{-\alpha}},
        \end{equation}
    while we interpolate the high-frequency one to get 
    \begin{align}
        \int_\delta^T \| \theta_{>N}^\nu (\tau) \|_{\dot H^{\gamma - \alpha}}^{2} \, d\tau&\leq C \int_\delta^T\| \theta^\nu_{>N} (\tau)\|^{2\frac{\alpha}{\alpha+\gamma}}_{\dot H^{- \alpha}}  \| \theta^\nu (\tau)\|^{2\frac{\gamma}{\alpha+\gamma}}_{\dot H^{ \gamma}}\,d\tau\\
        &\leq C \left(\int_\delta^T\| \theta^\nu_{>N} (\tau)\|^{2}_{\dot H^{- \alpha}} \,d\tau\right)^\frac{\alpha}{\alpha+\gamma} \left(\int_\delta^T\| \theta^\nu (\tau)\|^{2}_{\dot H^{ \gamma}} \,d\tau\right)^\frac{\gamma}{\alpha+\gamma},\label{est ugly 1}
    \end{align}
    where the last bound follows by the H\"older inequality. Recall that we are considering a sequence $\{\theta^\nu\}_\nu$ of solutions to \eqref{g-SQG diss} satisfying the higher-order bound \eqref{hob}. In particular, since we are assuming $\{f^\nu\}_\nu$ to be bounded in $L^{\frac{\alpha+\gamma}{\gamma}}([0,T];\dot H^{-\alpha}(\T^2))$ and $L^\infty([0,T];\dot H^{-\gamma}(\T^2))$ if $\alpha\leq \gamma$ and $\alpha>\gamma$ respectively, by $(c)$ from Proposition \ref{P:leray exist} we deduce 
    \begin{equation}
        \label{hob uniform} 
        \sup_{\nu \in (0,1)} \nu^{\frac{\alpha+\gamma}{\gamma}} \int^T_\delta \|\theta^\nu(\tau)\|_{\dot H^{\gamma}}^2 \, d\tau<\infty.
    \end{equation}
Therefore, \eqref{est ugly 1} becomes 
    \begin{equation}\label{est ugly 2}
         \int_\delta^T \| \theta_{>N}^\nu (\tau) \|_{\dot H^{\gamma - \alpha}}^{2} \, d\tau \leq C \nu^{-1} \left(\int_\delta^T\| \theta^\nu_{>N} (\tau)\|^{2}_{\dot H^{- \alpha}} \,d\tau\right)^\frac{\alpha}{\alpha+\gamma}.
    \end{equation}
    Summing up, by plugging \eqref{est ugly 3} and \eqref{est ugly 2} into \eqref{dissipation split low and high freq}  we achieve 
    \begin{equation}
        \label{est ugly 4}
         \nu \int_\delta^T \| \theta^\nu (\tau) \|_{\dot H^{\gamma - \alpha}}^{2} \, dt\leq C \left( \nu N^{2\gamma}+ \left(\int_\delta^T\| \theta^\nu_{>N} (\tau)\|^{2}_{\dot H^{- \alpha}} \,d\tau\right)^\frac{\alpha}{\alpha+\gamma} \right)\qquad \forall N>0,
    \end{equation}
    for a constant $C<\infty$ which does not depend on either $N$ or $\nu$.  Now, if $\{\theta^\nu\}_\nu$ is strongly compact in $L^2([0,T];\dot H^{- \alpha} (\T^2))$ we have 
    $$
    \lim_{N\rightarrow \infty}\sup_{\nu \in (0,1)} \int_\delta^T\| \theta^\nu_{>N} (\tau)\|^{2}_{\dot H^{- \alpha}} \,d\tau=0.
    $$
    Therefore, by first letting $\nu\rightarrow 0^+$ in \eqref{est ugly 4}  and then $N\rightarrow \infty$, we deduce 
    $$
    \lim_{\nu\rightarrow 0^+}   \nu \int_\delta^T \| \theta^\nu (\tau) \|_{\dot H^{\gamma - \alpha}}^{2} \, d\tau=0.
    $$
    The proof is concluded by letting $\nu\rightarrow 0^+$ in \eqref{split diss for short and positive times} and by the arbitrariness of  $\eps>0$.
\end{proof}

We conclude this section by giving an improved version of \cref{T:main diss}, identifying the relevant scales at which the strong compactness becomes effective on the dissipation. For scaling reasons due to the dissipative term $\nu(-\Delta)^\gamma$, this corresponds to length scales $\ell_\nu\sim \nu^{2\gamma}$. In terms of  frequencies this is $N_\nu\sim \nu^{-\frac{1}{2\gamma}}$. In fact, compactness at these frequencies becomes equivalent to no anomalous dissipation of the Hamiltonian. The terminology ``compactness at a given frequency'' is motivated by the fact that a bounded sequence $\{g^j\}_j\subset \ham(\T^2)$ is strongly compact in $\dot H^{-\alpha}(\T^2)$ if and only if 
$$
\lim_{N\rightarrow \infty} \sup_{j\geq 1}\|g^j_{>N}\|_{\dot H^{-\alpha}}=0,
$$
where we are using the notation from \eqref{freq cuts} to denote the Fourier projection on high frequencies. The theorem below is a generalization of \cite{DLP25}*{Theorem 3.3}.

\begin{theorem}[Frequency-dependent version]
    \label{T:compact and diss at frequencies}
   Let $\alpha\in(0,1]$ and $\gamma \in (0,\infty)$ be arbitrary. Let $\{\theta_0^{\nu}\}_\nu$ be a sequence of zero-average initial data and $\{f^\nu\}_\nu$ be a sequence of external forcing with zero spatial average for almost all times. Assume that 
\begin{itemize}
    \item[$(I)$] $\{\theta_0^{\nu}\}_\nu$ is strongly compact in $\dot H^{-\alpha}(\T^2)$;
    \item[$(II)$] $\{f^{\nu}\}_\nu$ stays bounded in $L^{\frac{\alpha+\gamma}{\gamma}}([0,T];\dot H^{-\alpha}(\T^2))$ if $\alpha\leq \gamma$, while $\{f^{\nu}\}_\nu$ is bounded in $L^\infty([0,T];\dot H^{-\gamma}(\T^2)) $ if $\alpha>\gamma$.
\end{itemize}
Let $\{\theta^\nu\}_{\nu}$ be any sequence of the weak solutions to \eqref{g-SQG diss} constructed in Proposition \ref{P:leray exist}. Then, setting $N_\nu:=\nu^{-\frac{1}{2\gamma}}$, it holds
\begin{equation}\label{equivalence}
    \lim_{\nu\rightarrow 0^+} \int_0^T \|\theta^\nu_{>\lambda  N_\nu} (\tau)\|^2_{\dot H^{-\alpha}} \,d\tau =0\quad \forall \lambda>0\quad \Longleftrightarrow\quad \lim_{\nu\rightarrow 0^+} \nu \int_0^T\|\theta^\nu(\tau)\|^2_{\dot H^{\gamma-\alpha}}\, d\tau=0.
       \end{equation}
\end{theorem}

\begin{proof}
  We start by proving the left-to-right implication. Let $\eps,\lambda>0$. By Proposition \ref{P:no instant dissipation} we find $\delta>0$ such that 
    \begin{equation}
        \label{dissipation small for small times 2}
        \sup_{\nu \in (0,1)} \nu\int_0^\delta \| \theta^\nu (\tau) \|_{\dot H^{\gamma - \alpha}}^{2} \, d\tau\leq \eps.
    \end{equation}
    Choose $N=\lambda N_\nu$ in \eqref{est ugly 4} to get 
  $$
   \nu \int_\delta^T \| \theta^\nu (\tau) \|_{\dot H^{\gamma - \alpha}}^{2} \, d\tau\leq C \left( \lambda^{2\gamma}+ \left(\int_\delta^T\| \theta^\nu_{> \lambda N_\nu} (\tau)\|^{2}_{\dot H^{- \alpha}} \,d\tau\right)^\frac{\alpha}{\alpha+\gamma} \right)\qquad \forall \nu \in (0,1).
  $$
  Then, by splitting the dissipation into times $(0,\delta)$ and $(\delta,T)$, our assumption shows 
  $$
  \limsup_{\nu\rightarrow 0^+} \nu \int_0^T \| \theta^\nu (\tau) \|_{\dot H^{\gamma - \alpha}}^{2} \, d\tau\leq \eps+ C \lambda^{2\gamma},
  $$
 which then must necessarily vanish by the arbitrariness of  $\eps,\lambda>0$.  The proof of the right-to-left implication is a direct consequence of 
   $$
   \int_0^T \| \theta^\nu_{>N} (\tau)\|^{2}_{\dot H^{- \alpha}} \,d\tau \leq N^{-2\gamma}  \int_0^T\|\theta^\nu(\tau)\|^2_{\dot H^{\gamma- \alpha}}\, d\tau \qquad \forall N,\,\nu>0.
   $$
\end{proof}

\section{Global existence}\label{S:global existence}

This section is dedicated to the proof of Theorem \ref{T:global-existence}. Let $\{\theta^\nu_0\}_\nu$ and $\{f^\nu\}_\nu$ be  smooth approximations of $\theta_0\in L^{p_\alpha} (\T^2)$ and $f\in L^1_{\rm loc}([0,\infty);L^{p_\alpha}(\T^2))$ respectively, converging strongly in their own norms. In order to construct a global-in-time weak solution to \eqref{g-SQG} we consider its viscous regularization \eqref{g-SQG diss} for a large enough\footnote{In view of \cite{MX12}, it is enough to take any $\gamma\geq 1-\alpha$.} power $\gamma$ of the Laplacian that guarantees global well-posedness. Then, by standard estimates on transport-diffusion equation with a divergence-free vector field (see Proposition \ref{P:adv diff Lp bound} below) we get
$$
\|\theta^\nu(t)\|_{L^{p}}\leq \|\theta^\nu_0\|_{L^p}+\int_0^t \|f^\nu(\tau)\|_{L^{p}}\,d\tau \qquad \forall t>0, \, p\in [1,\infty].
$$
As $\nu\rightarrow 0^+$, the right-hand-side stays bounded for $p=p_\alpha$. It follows that $\{\theta^\nu\}_\nu$ is bounded in $L^\infty_{\rm loc}([0,\infty);L^{p_\alpha}(\T^2))$. Recall that the value of $p_\alpha=\frac{2}{1+\alpha}$ is critical, i.e. the embedding $L^{p_\alpha}(\T^2)\subset\dot H^{-\alpha}(\T^2)$ is only continuous. Therefore, although this allows to extract a subsequence $\theta^\nu\overset{*}{\rightharpoonup}\theta$ in $L^\infty_{\rm loc}([0,\infty);\dot H^{-\alpha}(\T^2))$, this is not enough to pass to the limit into the nonlinear term which, in view of Remark \ref{R:nonlinearity continuous}, requires the strong compactness in the latter space. To overcome this issue we leverage  on the strong compactness of $\{\theta^\nu_0\}_\nu$ and $\{f^\nu\}_\nu$ in $L^{p_\alpha}(\T^2)$ and $L^1_{\rm loc}([0,\infty);L^{p_\alpha}(\T^2))$ respectively. This allows to split the sequence $\{\theta^\nu\}_\nu$ into an  ``arbitrarily small'' part and a ``regular'' one. More precisely, for any $\eps>0$, we will provide a decomposition 
$$
\theta^\nu=\theta^{\nu,\eps}_{\rm smll}+ \theta^{\nu,\eps}_{\rm reg}
$$
 such that 
\begin{equation}\label{decomp sol}
\| \theta^{\nu,\eps}_{\rm smll} (t)\|_{L^{p_\alpha}}\leq \eps \qquad \text{and} \qquad \| \theta^{\nu,\eps}_{\rm reg} (t)\|_{L^{2}}\leq R_\eps\qquad \forall \nu>0, \, t\in [0,T],
\end{equation}
for a, maybe large, constant $R_\eps<\infty$. It is well-known that decompositions like \eqref{decomp sol} are related to equi-integrability. See Remark \ref{R: concentration comp} below and also \cites{BGM25,BCC22} where the same idea has been used for similar purposes. Although it certainly does not imply strong compactness of $\{\theta^\nu\}_\nu$ in $L^\infty_{\rm loc}([0,\infty);L^{p_\alpha}(\T^2))$, such decomposition is enough for the one in $L^\infty_{\rm loc}([0,\infty);\dot H^{-\alpha}(\T^2))$ as shown in the next lemma. The lemma will be then applied to $\{\theta^\nu(t)\}_\nu$ for all $t\geq 0$.

\begin{lemma}[Concentration compactness]\label{InterpolatoryLem}
    Let $\alpha\in (0,1)$ and $p_\alpha:=\frac{2}{1+\alpha}$. Let $\{g^j\}_{j}$ be such that $g^j\overset{*}{\rightharpoonup} g$ in $\dot H^{-\alpha}(\T^2)$. Assume that, for every $\varepsilon > 0$, there exists a decomposition $g^j = g^{j,\eps}_{\rm smll} +  g^{j,\eps}_{\rm reg}$ and a constant $R_\eps<\infty$ such that 
    \begin{equation}\label{split small plus reg}
        \| g^{j,\eps}_{\rm smll} \|_{L^{p_\alpha}} \leq \varepsilon \qquad \text{and}\qquad 
        \| g^{j,\eps}_{\rm reg} \|_{L^2} \leq R_\eps \qquad \forall j\geq 1.
    \end{equation}
    Then $g^j\rightarrow g$ in $\dot H^{-\alpha}(\T^2)$.
\end{lemma}

\begin{proof}
    In the notation \eqref{freq cuts}, the claimed strong convergence is equivalent to 
    \begin{equation}
        \label{gn small on high freq}
        \lim_{N\rightarrow \infty}\sup_{j\geq 1}\|g^j_{>N}\|_{\dot H^{-\alpha}}=0.
    \end{equation}
    Let $\eps>0$. By the decomposition assumption and the continuity of the embedding $L^{p_\alpha}(\T^2)\subset \dot H^{-\alpha}(\T^2)$ we have
    \begin{align}
        \|g^j_{>N}\|_{\dot H^{-\alpha}}&\leq \|(g^{j,\eps}_{\rm smll})_{>N}\|_{\dot H^{-\alpha}} + \|(g^{j,\eps}_{\rm reg})_{>N}\|_{\dot H^{-\alpha}}\\
        &\leq  \| g^{j,\eps}_{\rm smll} \|_{\dot H^{-\alpha}} + N^{-\alpha} \| g^{j,\eps}_{\rm reg} \|_{L^2}\\
         &\leq  C \| g^{j,\eps}_{\rm smll} \|_{L^{p_\alpha}} + N^{-\alpha}R_\eps\\
         &\leq  C \eps + N^{-\alpha} R_\eps.
    \end{align}
    Therefore 
    $$
    \limsup_{N\rightarrow \infty} \sup_{j\geq 1}\|g^j_{>N}\|_{\dot H^{-\alpha}}\leq C \eps,
    $$
    from which \eqref{gn small on high freq} follows by the arbitrariness of $\eps>0$.
\end{proof}

\begin{remark}\label{R: concentration comp}
    Lemma \ref{InterpolatoryLem} is a particular case of the more general ``concentration compactness principle'' à la Lions \cites{PLL1,PLL2}. The concentration compactness principle guarantees the strong compactness of $\{g^j\}_j$ in $\dot H^{-\alpha}(\T^2)$ as soon as $\{|g^j|^{p_\alpha}\}_j$ does not generate atomic concentrations in $L^1(\T^2)$. Since $p_\alpha<2$ for any $\alpha\in (0,1)$, one can check that the assumption \eqref{split small plus reg} implies that $\{|g^j|^{p_\alpha}\}_j$ is weakly compact in $L^1(\T^2)$, and therefore no concentration can occur. However, although excluding atomic concentrations would be enough, there is no currently known way to propagate the nonatomic condition from the initial data. It is for this reason that we are working under the stronger assumption \eqref{split small plus reg} which, as we will see below in the proof of Theorem \ref{T:global-existence}, can be shown to propagate from  $\theta_0$, uniformly in time.
\end{remark}

\begin{remark}
    \label{R: no subseq}
    We emphasize that Lemma \ref{InterpolatoryLem} proves that the whole sequence converges strongly. This is important as Lemma \ref{InterpolatoryLem} will be applied to  $\{\theta^\nu(t)\}_\nu$ for all $t\geq 0$ and, therefore, passing to a different subsequence depending on each time $t$ would be a problem. 
\end{remark}

We are now ready to prove our second, and last, main theorem. 

\begin{proof}[Proof of Theorem \ref{T:global-existence}]
We break the proof down into steps.

\underline{\textsc{Step 1}}: Regularization. 

We will regularize \eqref{g-SQG} by adding viscosity and mollifying the data $\theta_0$ and $f$. More precisely, since $\theta_0\in L^{p_\alpha}(\T^2)$ and $f\in L^1_{\rm loc}([0,\infty);L^{p_\alpha}(\T^2))\cap L^{1+\alpha}_{\rm loc}([0,\infty);\dot H^{-\alpha}(\T^2))$, we find $\{\theta^\nu\}_\nu\subset C^\infty(\T^2)$ and $\{f^\nu\}_\nu\subset C^\infty(\T^2\times [0,\infty))$ such that 
\begin{align}
    \|\theta^\nu_0\|_{L^{p_\alpha}}&\leq  \|\theta_0\|_{L^{p_\alpha}}, \label{unif Lp bound data}\\
    \int_0^T \|f^\nu(t)\|^{1+\alpha}_{\dot H^{-\alpha}}\,dt & \leq  \int_0^T \|f(t)\|^{1+\alpha}_{\dot H^{-\alpha}}\,dt\label{unif ham bound force}
\end{align}
and 
\begin{equation}
    \label{unif Lp bound force}
     \int_0^T \|f^\nu(t)\|_{L^{p_\alpha}}\,dt \leq  \int_0^T \|f(t)\|_{L^{p_\alpha}}\,dt
\end{equation}
for all $\nu>0$ and $T<\infty$. Moreover
\begin{equation}
    \label{in data conv strong Lp}
    \theta^\nu_0\rightarrow \theta_0 \qquad \text{in } L^{p_\alpha}(\T^2)
\end{equation}
and 
\begin{equation}
    \label{force conv strong Lp}
    f^\nu\rightarrow f \qquad \text{in } L^1_{\rm loc}([0,\infty);L^{p_\alpha}(\T^2)).
\end{equation}
For a sequence $\nu\rightarrow 0^+$, let $\{\theta^\nu\}_\nu$ be the unique smooth solutions\footnote{By \cite{MX12}, for any smooth initial datum and with no external forcing, the system \eqref{g-SQG diss} admits a unique global-in-time smooth solution for $\gamma=1-\alpha$. In particular, the same is true for any $\gamma\geq 1-\alpha$. A smooth external body force can be incorporated in the energy estimates argument from \cite{MX12}.} to \eqref{g-SQG diss} for a given $\gamma\in [1-\alpha,1]$. For simplicity, let us fix $\gamma=1$. By smoothness and Proposition \ref{P:adv diff Lp bound} we deduce 
\begin{equation}
    \label{ham cons in proof 2}
    \frac12  \|\theta^\nu(t)\|^2_{\dot H^{-\alpha} } + \nu \int_{0}^t\|\theta^\nu(\tau)\|^2_{\dot H^{1-\alpha}}\,d\tau =   \frac12  \|\theta^\nu_0\|^2_{\dot H^{-\alpha} } + \int_0^t\langle f^\nu (\tau),\theta^\nu(\tau)\rangle_{\dot H^{-\alpha}}\,d\tau
\end{equation}
and
\begin{equation}
    \label{Lp bound unif in time}
     \|\theta^\nu(t)\|_{L^{p_\alpha} } \leq   \|\theta^\nu_0\|_{L^{p_\alpha} }+\int_0^t\|f^\nu(\tau)\|_{L^{p_\alpha}}\,d\tau
\end{equation}
for all $\nu,t>0$. Note that, as a direct consequence of \eqref{unif Lp bound data}, \eqref{unif Lp bound force} and \eqref{Lp bound unif in time} we get
\begin{equation}
    \label{sequence of sol bound Lp}
    \{\theta^\nu\}_\nu\subset L^\infty_{\rm loc}([0,\infty); L^{p_\alpha}(\T^2))\qquad \text{is bounded}.
\end{equation}

\underline{\textsc{Step 2}}: Decomposition of $\theta^\nu$. 

We want to provide a decomposition of $\theta^\nu$ as in \eqref{decomp sol}. To do that, let $\eps>0$ and $T<\infty$. By \eqref{in data conv strong Lp} and \eqref{force conv strong Lp},  via a suitable regularization\footnote{A sequence $\{g^j\}_j\subset L^p$ is strongly compact if and only if $\sup_j \|g^j - g^j*\rho_\eps\|_{L^p}\rightarrow 0$ as $\eps\rightarrow 0^+$. Moreover, to ensure \eqref{decomp force}, the regularization has to depend on $T$ as well. This was not explicitly written so as not to make the notation too heavy.} of $\theta^\nu_0$ and $f^\nu$, we find $\theta^{\nu,\eps}_0$ and $f^{\nu,\eps}$ such that 
\begin{equation}
    \label{decomp init data}
    \|\theta^\nu_0 - \theta^{\nu,\eps}_0\|_{L^{p_\alpha}}\leq \frac{\eps}{2}\qquad \text{and} \qquad  \| \theta^{\nu,\eps}_0\|_{L^{2}}\leq \frac{R_\eps}{2}\qquad \forall \nu>0
\end{equation}
and 
\begin{equation}
    \label{decomp force}
    \int_0^T\|(f^{\nu} - f^{\nu,\eps})(t)\|_{L^{p_\alpha}}\,dt\leq \frac{\eps}{2}\qquad \text{and} \qquad   \int_0^T\|f^{\nu,\eps}(t)\|_{L^{2}}\,dt\leq \frac{R_\eps}{2}\qquad \forall \nu>0,
\end{equation}
for a sufficiently large constant $R_\eps<\infty$. Therefore, by setting $\theta^{\nu,\eps}_{0,\rm smll}:=\theta^{\nu}_{0}- \theta^{\nu,\eps}_{0}$ and  $\theta^{\nu,\eps}_{0,\rm reg}:= \theta^{\nu,\eps}_{0}$, and similarly for $f^\nu$, we have 
$$
\theta^\nu_0 = \theta^{\nu,\eps}_{0,\rm smll} + \theta^{\nu,\eps}_{0,\rm reg}\qquad \text{and}\qquad f^\nu = f^{\nu,\eps}_{\rm smll} + f^{\nu,\eps}_{\rm reg}.
$$
Consequently, we let $\theta^{\nu,\eps}_{\rm smll}$ and $\theta^{\nu,\eps}_{\rm reg}$ to be the unique smooth solutions to the linear systems
$$
\left\{\begin{array}{l}
\partial_t\theta^{\nu,\eps}_{\bullet} + u^\nu\cdot\nabla\theta^{\nu,\eps}_{\bullet} - \nu \Delta \theta^{\nu,\eps}_{\bullet} = f^{\nu,\eps}_{\bullet}\\
\theta^{\nu,\eps}_{\bullet}(\cdot,0)=\theta^{\nu,\eps}_{0,\bullet}
\end{array}\right. \qquad \bullet=\rm smll, \rm reg
$$
on $\T^2\times [0,\infty)$. By linearity, it must be $\theta^\nu = \theta^{\nu,\eps}_{\rm smll} +\theta^{\nu,\eps}_{\rm reg}$. By Proposition \ref{P:adv diff Lp bound} we deduce 
$$
    \|\theta^{\nu,\eps}_{\bullet}(t)\|_{L^{p} } \leq   \|\theta^{\nu,\eps}_{0,\bullet}\|_{L^{p} }+\int_0^t\|f^{\nu,\eps}_{\bullet}(\tau)\|_{L^{p}}\,d\tau \qquad \forall \nu, \,t>0, \, p\in [1,\infty],
$$
for both $\bullet = \rm smll, \rm reg$. Therefore, by \eqref{decomp init data} and \eqref{decomp force} we deduce
\begin{equation}
    \label{final decomp solut}
    \|\theta^{\nu,\eps}_{\rm smll}(t)\|_{L^{p_\alpha}}\leq \eps \qquad \text{and}\qquad  \|\theta^{\nu,\eps}_{\rm reg}(t)\|_{L^{2}}\leq R_\eps\qquad \forall \nu>0,\, t\in [0,T].
\end{equation}

\underline{\textsc{Step 3}}: Strong compactness. 

Let $p'_\alpha$ be the H\"older conjugate of $p_\alpha$. Let $T<\infty$ and $\psi \in L^{p_\alpha'}(\T^2)$. We claim that the sequence of functions 
$$
t\mapsto \int_{\T^2}\theta^\nu(x,t)\psi(x)\,dx
$$
is equi-continuous on $[0,T]$. First of all, by \eqref{sequence of sol bound Lp} it is enough to prove the claim for $\psi \in C^\infty (\T^2)$ by density. Therefore, let $\psi \in C^\infty (\T^2)$. We need to prove that, for any $\eps>0$ there exists a $\delta>0$ such that 
\begin{equation}
    \label{equicont claim}
    |t-s|<\delta \qquad \Longrightarrow \qquad \sup_{\nu>0}\left|\int_{\T^2} \left(\theta^\nu(x,t) - \theta^\nu(x,s)\right)\psi(x)\,dx\right| <\eps,
\end{equation}
for $t,s\in [0,T]$. Since $\theta^\nu$ is smooth, we can write 
$$
\theta^\nu(x,t) - \theta^\nu(x,s)=\int_s^t \partial_\tau \theta^\nu(x,\tau)\,d\tau,
$$
from which
$$
\int_{\T^2} \left(\theta^\nu(x,t) - \theta^\nu(x,s)\right)\psi(x)\,dx= \int_s^t\int_{\T^2} \left(\theta^\nu u^\nu\cdot \nabla \psi +\nu \theta^\nu\Delta \psi +f^\nu\psi\right)\,dxd\tau.
$$
Therefore, since  we can take $\nu<1$, by the identity \eqref{reassembling nonlinearity} and Proposition \ref{P:nonlin continuous} we deduce
\begin{align}
    \left|\int_{\T^2} \left(\theta^\nu(x,t) - \theta^\nu(x,s)\right)\psi(x)\,dx\right|&\leq C \int_s^t \left(\|\theta^\nu(\tau)\|^2_{\dot H^{-\alpha}} +\|\theta^\nu(\tau)\|_{\dot H^{-\alpha}} + \|f^\nu(\tau)\|_{\dot H^{-\alpha}} \right)\,d\tau\\
    &\leq C\left(|t-s| +\sup_{\nu>0}\int_s^t \|f^\nu(\tau)\|_{L^{p_\alpha}}\,d\tau\right),     \label{last estimate for equicont}
\end{align}
where to obtain the last inequality we have used \eqref{sequence of sol bound Lp}, since $t,s\in [0,T]$ and $T<\infty$. By \eqref{force conv strong Lp}, the last term in \eqref{last estimate for equicont} can be made arbitrarily small if $|t-s|\ll 1$, thereby proving \eqref{equicont claim}. The proof of the claim is concluded. 

The equi-continuity together with \eqref{sequence of sol bound Lp} allows to apply the  Ascoli--Arzelà compactness theorem\footnote{The domain $[0,\infty)$ is locally compact and separable while the target space $L^{p_\alpha}(\T^2)$ is compact when endowed with the weak topology.} to find an element $\theta\in C^0_{\rm w} ([0,\infty); L^{p_\alpha}(\T^2))$ such that 
\begin{equation}
    \label{asc arz}
    \lim_{\nu\rightarrow 0^+} \sup_{t\in [0,T]} \left|\int_{\T^2} \left(\theta^\nu(x,t)-\theta(x,t)\right)\psi (x)\,dx \right|=0\qquad \forall T<\infty, \,  \psi \in L^{p_\alpha'}(\T^2),
\end{equation}
up to a nonrelabelled subsequence. In particular 
\begin{equation}\label{weak Lp conv for all times}
    \theta^\nu(t)\rightharpoonup \theta(t) \qquad \text{in }  L^{p_\alpha}(\T^2),  \text{ for all } t\geq 0.
    \end{equation}
Since $ L^{p_\alpha}(\T^2)\subset  \dot H^{-\alpha}(\T^2)$, then $\theta^\nu(t)\overset{*}{\rightharpoonup}\theta(t)$ in $ \dot H^{-\alpha}(\T^2)$ as well. Therefore, by \eqref{final decomp solut} we can invoke Lemma \ref{InterpolatoryLem} (see also Remark \ref{R: no subseq}) to deduce 
\begin{equation}
    \label{strong con ham for all times}
    \theta^\nu(t)\rightarrow \theta(t) \qquad \text{in }  \dot H^{-\alpha}(\T^2),  \text{ for all } t\geq 0.
\end{equation}

\underline{\textsc{Step 4}}: Conclusion.

We are now ready to conclude the proof of the theorem by putting together all the above considerations. Note that $\{\theta^\nu\}_\nu$ stays bounded in $L^\infty_{\rm loc}([0,\infty);\dot H^{-\alpha}(\T^2))$ by \eqref{sequence of sol bound Lp}. Then, the pointwise-in-time convergence \eqref{strong con ham for all times} yields to $\theta^\nu\rightarrow \theta$ in $L^2_{\rm loc}([0,\infty);\dot H^{-\alpha}(\T^2))$ by the dominated convergence theorem. In view of Remark \ref{R:nonlinearity continuous}, this allows to pass to the limit into the nonlinearity. Together with \eqref{in data conv strong Lp} and \eqref{force conv strong Lp}, this shows that 
\begin{equation}\label{weak cont of sol}
\theta\in C^0_{\rm w} ([0,\infty); L^{p_\alpha}(\T^2)) \subset C^0_{\rm w} ([0,\infty); \dot H^{-\alpha}(\T^2))
\end{equation}
is a weak solution to \eqref{g-SQG} in the sense of Definition \ref{def.sqg-solutions}. The inequality \eqref{Lp inequality} follows by using \eqref{in data conv strong Lp}, \eqref{force conv strong Lp} and \eqref{weak Lp conv for all times} into \eqref{Lp bound unif in time} thanks to the lower semicontinuity of the $L^{p_\alpha}(\T^2)$ norm under weak convergence. Finally, by \eqref{in data conv strong Lp}, \eqref{force conv strong Lp}, \eqref{asc arz} and \eqref{strong con ham for all times} we have 
\begin{equation}
\|\theta^\nu(t)\|_{\dot H^{-\alpha}}\rightarrow \|\theta(t)\|_{\dot H^{-\alpha}}
\end{equation}
and 
$$
\int_0^t \langle \theta^\nu(\tau),f^\nu(\tau)\rangle_{\dot H^{-\alpha}}\,d\tau \rightarrow \int_0^t \langle \theta(\tau),f(\tau)\rangle_{\dot H^{-\alpha}}\,d\tau
$$
for all $t\geq 0$, with $\theta(0)=\theta_0$. Therefore, since $\{f^\nu\}_\nu\subset L^{1+\alpha}_{\rm loc}([0,\infty);\dot H^{-\alpha}(\T^2))$ stays bounded by \eqref{unif ham bound force}, we can let $\nu\rightarrow 0^+$ in \eqref{ham cons in proof 2} and deduce the validity of \eqref{hamiltonian balance} by Theorem \ref{T:main diss}. Note that we are allowed to use Theorem \ref{T:main diss} since we are considering the unique smooth solution to \eqref{g-SQG diss} for $\gamma=1$ which, as such, it must enjoy all the properties listed in Proposition \ref{P:leray exist}. To conclude, \eqref{hamiltonian balance} together with the weak continuity \eqref{weak cont of sol} yields to the strong continuity in time, i.e.  $\theta\in C^0 ([0,\infty); \dot H^{-\alpha}(\T^2))$.
\end{proof}

\begin{remark}\label{R:no L2 force}
    In the proof of Theorem \ref{T:global-existence}, the assumption $f\in L^{1+\alpha}_{\rm loc}([0,\infty);\dot H^{-\alpha}(\T^2))$ has only been used in the last Step 4 to be able to use Theorem \ref{T:main diss} in order to exclude anomalous dissipation and deduce the Hamiltonian balance \eqref{hamiltonian balance} at the inviscid level. Therefore, by following the very same proof, all the theses but \eqref{hamiltonian balance} in Theorem \ref{T:global-existence} stay true even without the assumption $f\in L^{1+\alpha}_{\rm loc}([0,\infty);\dot H^{-\alpha}(\T^2))$. It might be of course possible that, at least in the case $\gamma\geq \alpha$, Theorem \ref{T:main diss} is still valid under the weaker assumption that the forcing $\{f^\nu\}_\nu\subset L^1([0,T];\dot H^{-\alpha}(\T^2))$ is bounded. To do that, one should prove the same higher-order bound \eqref{hob} under this weaker integrability in time of the forcing. Although most of the apriori estimates obtained in the proof of Proposition \ref{P:leray exist} stay true for $f^\nu\in L^1([0,T];\dot H^{-\alpha}(\T^2))$, some others seem to be more involved\footnote{For instance, it is not clear to us how to deal with \eqref{estimate not generalized}.}. Therefore, it is not clear  whether Theorem \ref{T:main diss} is still true by weakening the, uniform-in-viscosity, time integrability of the forcing to $L^1([0,T])$. However, in view of these considerations, Theorem \ref{T:global-existence} should still be true, including the Hamiltonian balance \eqref{hamiltonian balance}, by removing the assumption $f\in L^{1+\alpha}_{\rm loc}([0,\infty);\dot H^{-\alpha}(\T^2))$. For instance, instead of the vanishing viscosity one, it should be enough to consider a regularization\footnote{For instance, by suitably regularizing the nonlinearity only.} of \eqref{g-SQG} that does not introduce a ``dissipative defect'' in \eqref{hamiltonian balance}. Since this goes somehow in a different direction with respect to the main spirit of the current paper, it wont be handled here.
\end{remark}

The next proposition is quite standard. We include the proof for the reader's convenience.

\begin{proposition}[$L^p$ bounds for transport-diffusion]
    \label{P:adv diff Lp bound}
   Let $u$ be a smooth divergence-free vector field. Let $\rho_0,f$ and $\rho$ be smooth and such that 
   $$
   \left\{\begin{array}{l}
\partial_t\rho + u\cdot\nabla\rho - \nu \Delta \rho= f\\
\rho(\cdot,0)=\rho_{0}
\end{array}\right. 
   $$
   on $\T^2\times [0,\infty)$. Then, for any $p\in [1,\infty]$, it holds 
   \begin{equation}\label{stab Lp est}
   \|\rho (t)\|_{L^p}\leq \|\rho_0\|_{L^p} +\int_0^t  \|f (\tau)\|_{L^p}\,d\tau \qquad \forall t\geq 0.
    \end{equation}
\end{proposition}
\begin{proof}
    Let $\beta:\R\rightarrow \R$ be convex and smooth. Multiplying the equation by $\beta'(\rho)$ and integrating in the space variable we get 
    \begin{align}
    \frac{d}{dt}\int_{\T^2} \beta(\rho(x,t))\,dx&\leq  \frac{d}{dt}\int_{\T^2} \beta(\rho(x,t))\,dx +\nu \int_{\T^2}\beta''(\rho(x,t))|\nabla \rho(x,t)|^2\,dx\\
    &=\int_{\T^2} \beta'(\rho(x,t))f(x,t)\,dx.
        \end{align}
        Let $p\in (1,\infty)$. By a standard approximation argument we can chose $\beta(\cdot)=|\cdot|^p$ to get 
        $$
        \frac{d}{dt}\|\rho(t)\|^p_{L^p}\leq p \int_{\T^2} |\rho(x,t)|^{p-1}|f(x,t)|\,dx \leq p \left(\|\rho(t)\|_{L^p}^p\right)^{\frac{p-1}{p}} \|f(t)\|_{L^p}\qquad \forall t>0.
        $$
        Then, by the comparison principle, $\|\rho(t)\|^{p}_{L^p}$ must stay below the solution
        $$
        \left(\|\rho_0\|_{L^p} +\int_0^t\|f(\tau)\|_{L^p}\,d\tau\right)^p
        $$
        of the corresponding ODE. This proves \eqref{stab Lp est} for any $p\in (1,\infty)$. The cases $p=1$ and $p=\infty$ can be obtained by letting $p$ converge to the endpoints.
\end{proof}

\bibliographystyle{plain} 
\bibliography{biblio}

\end{document}